[1994/12/01]
\documentclass[10pt, reqno]{amsart}
\usepackage{a4wide}
\usepackage[english, activeacute]{babel}
\usepackage{amsmath,amsthm,amsxtra}
\usepackage{epsfig}
\usepackage{amssymb}
\usepackage{latexsym}
\usepackage{amsfonts}
\usepackage{hyperref}
\title{Inelastic character of solitons of slowly varying gKdV equations}
\author{Claudio Mu\~noz}
\address{Departamento de Ingenier\'ia Matem\'atica y CMM, Universidad de Chile, Casilla 170-3, Correo 3, Santiago \\ Chile}
\email{cmunoz@dim.uchile.cl}
\date{July, 2011}
\subjclass[2000]{Primary 35Q51, 35Q53; Secondary 37K10, 37K40}
\keywords{gKdV equations, soliton dynamics, slowly varying medium}

\chardef\bslash=`\\ 

\newtheorem{thm}{Theorem}[section]

\newtheorem{lem}[thm]{Lemma}
\newtheorem{prop}[thm]{Proposition}

\theoremstyle{definition}

\theoremstyle{remark}
\newtheorem{rem}{Remark}[section]

\numberwithin{equation}{section}

\newcommand{\R}{\mathbb{R}}

\newcommand{\la}{\lambda}

\newcommand{\al}{\alpha}
\newcommand{\ga}{\gamma}

\def\bm{\left( \begin{array}{cc}}
\def\endm{\end{array}\right)}

 \providecommand{\abs}[1]{\lvert#1 \rvert}

\newcommand{\ve}{\varepsilon}

\newcommand{\be}{\begin{equation}}
\newcommand{\ee}{\end{equation}}
\newcommand{\ba}{\begin{equation*}}
\newcommand{\ea}{\begin{equation*}}
\newcommand{\bea}{\begin{eqnarray}}
\newcommand{\eea}{\end{eqnarray}}
\newcommand{\bee}{\begin{eqnarray*}}
\newcommand{\eee}{\end{eqnarray*}}
\newcommand{\ben}{\begin{enumerate}}
\newcommand{\een}{\end{enumerate}}
\newcommand{\nonu}{\nonumber}
\newcommand{\eval}[2][\right]{\relax
  \ifx#1\right\relax \left.\fi#2#1\rvert}

\let\abs=\envert

\begin{document}
\begin{abstract}
In this paper we study \emph{soliton-like} solutions of the variable coefficients, subcritical gKdV equation   
$$
u_t + (u_{xx} -\la u  + a(\ve x) u^m )_x =0,\quad \hbox{ in } \quad \R_t\times \R_x,  \quad m=2,3 \hbox{ and } 4,
$$
with $\la\geq 0$, $a(\cdot )\in (1,2)$ a strictly increasing, positive and asymptotically flat potential, and $\ve$ small enough. In \cite{Mu2,Mu3} it was proved the existence of a \emph{pure}, global in time, soliton $u(t)$ of the above equation, satisfying
$$
\lim_{t\to -\infty}\|u(t) - Q_1(\cdot -(1-\la)t) \|_{H^1(\R)} =0, \quad 0\leq \la<1,
$$
provided $\ve$ is small enough. Here $R(t,x) := Q_c(x-(c-\la)t)$ is the soliton of $R_t + (R_{xx} -\la R + R^m)_x =0$. In addition, there exists $\tilde \la \in (0,1)$ such that, for all $0<\la <1$ with $\la\neq \tilde \la$, the solution $u(t)$ satisfies
$$
 \sup_{t\gg \frac 1\ve }\|u(t) - \kappa(\la)Q_{c_\infty}(\cdot -\rho(t)) \|_{H^1(\R)} \lesssim \ve^{1/2}.
$$
Here  $\rho'(t) \sim (c_\infty(\la) -\la)$, with  $\kappa(\la) =2^{-1/(m-1)}$ and $c_\infty(\la)>\la$ in the case $0<\la<\tilde \la$ (\emph{refraction}), and $\kappa(\la) =1$ and $c_\infty(\la)<\la$ in the case $\tilde \la<\la<1$ (\emph{reflection}).   

In this paper we improve our preceding results by proving that the soliton is far from being pure as $t\to +\infty$. Indeed, we give a \emph{lower bound} on the defect induced by the potential $a(\cdot)$, for all $0<\la<1$, $\la\neq \tilde \la$. More precisely, one has
$$
\liminf_{t\to +\infty} \| u(t) -  \kappa_m(\la)Q_{c_\infty}(\cdot -\rho(t)) \|_{H^1(\R)} \gtrsim \ve^{1 +\delta},
$$
for any $\delta>0$ fixed. This bound clarifies the existence of a dispersive tail and the difference with the standard solitons of the constant coefficients, gKdV equation.
\end{abstract}
\maketitle \markboth{Inelastic solitons of gKdV equations} {Claudio Mu\~noz}
\renewcommand{\sectionmark}[1]{}

\section{Introduction and Main Results}

In this work we continue the study of the dynamics of a soliton-like solution for some generalized Korteweg-de Vries equations (gKdV), started in \cite{Mu2,Mu3}. In those papers the objective was the study of the global behavior of a \emph{generalized soliton solution} for the following subcritical, variable coefficients gKdV equation:
\be\label{aKdV0}
u_t + (u_{xx} -\la u  + a(\ve x) u^m )_x =0,\quad \hbox{ in } \quad \R_t\times \R_x,  \quad m=2,3 \hbox{ or } 4.
\ee
Here $u=u(t,x)$ is a real-valued function, $\ve>0$ is a small number, $\la\geq 0$ a fixed parameter, and the \emph{potential} $a(\cdot )$ a smooth, positive function satisfying some specific properties, see (\ref{ahyp}) below. 
 
\smallskip
 
This equation represents, in some sense, a simplified model of \emph{long dispersive waves in a channel with variable bottom}, which considers \emph{large} variations in the shape of the solitary wave. The primary physical model, and the dynamics of a generalized soliton-like solution, was formally described by Karpman-Maslov, Kaup-Newell  and Ko-Kuehl \cite{KM1,KN1,KK}, with further results by Grimshaw \cite{Gr1}, and Lochak \cite{Lo}. See \cite{Mu2, New} and references therein for a detailed physical introduction to this model. The main novelty in the works above cited was the discovery of a \emph{dispersive tail} behind the soliton, with small height but large width, as a consequence of the lack of conserved quantities such as mass or energy. However, \emph{no mathematical proof of this phenomenon has been given}. 

\smallskip

On the other hand, from the mathematical point of view, equation (\ref{aKdV0}) is a variable coefficients version of the standard gKdV equation
\be\label{gKdV}
u_t + (u_{xx}- \la u +u^m)_x =0, \; \hbox{ in } \; \R_t\times \R_x;\quad m\geq 2 \hbox{ integer}, \quad \la\geq0.
\ee
This last equation is famous due to the existence of localized, exponentially decaying, smooth solutions called \emph{solitons}. Given real numbers $x_0$ and $c>0$, solitons are solutions of (\ref{gKdV}) of the form
\be\label{(3)}
u(t,x):= Q_c(x-x_0-(c-\la)t), \quad  \hbox{ with } \quad Q_c(s):=c^{\frac 1{m-1}} Q(c^{1/2} s),
\ee  
and where $Q=Q_1$ is the unique --up to translations-- function satisfying the following, second order, nonlinear ordinary differential equation
$$Q'' -Q + Q^m =0, \quad Q>0, \quad Q\in H^1(\R).
$$
In this case, this solution belongs to the Schwartz class and it is explicit \cite{Mu3}. For $m=2,3,4$ solitons and the sum of solitons have been showed stable and asymptotically stable \cite{Benj,BSS,MMT,PW,We1}. 
In particular, if $c>\la$ the solution (\ref{(3)}) represents a \emph{solitary wave},\footnote{In this paper we will not do any distinction between soliton and solitary wave, unlike in the mathematical-physics literature.} of scaling $c$ and velocity $(c-\la)$, defined for all time, moving to the right without \emph{any change} in shape, velocity, etc. In other words, a soliton represents a \emph{pure}, traveling wave solution with \emph{invariant profile}. In addition, this equation has solitons with negative velocity, moving to the left, provided $c<\la$. Finally, for the case $c=\la$, one has a stationary soliton solution, $Q_\la(x-x_0)$. These last solutions do not exist in the standard, inviscid model of gKdV (namely when $\la=0$).  

\smallskip

Coming back to (\ref{aKdV0}), the corresponding Cauchy problem in $H^1(\R)$ has been considered in \cite{Mu2}. The proof of this result is an adaptation of the fundamental work of Kenig, Ponce and Vega \cite{KPV}, with the introduction of some new monotone quantities, in order to replace the lost conserved ones.

\smallskip

A fundamental question related to (\ref{gKdV}) is how to generalize a soliton-like solution to more complicated models. In \cite{BL}, the existence of soliton solutions for generalized KdV equations with suitable autonomous nonlinearities has been considered. However, less is known in the case of an inhomogeneous nonlinearity, such as equation (\ref{aKdV0}). In a general situation, no elliptic, time-independent ODE can be associated to the soliton, unlike the standard autonomous case studied in \cite{BL}. Therefore, other methods are needed.

\smallskip

The first mathematically rigorous results in the case of time and space dependent KdV and mKdV equations ($m=2$ and $m=3$ above) were proved by Dejak-Sigal and Dejak-Jonsson \cite{SJ,DS}. They studied the dynamics of a soliton for times of $O(\ve^{-1})$, and deduced dynamical laws which characterize the whole soliton dynamics up to some order of accuracy. More recently, Holmer  \cite{H} has improved the Dejak-Sigal results in the KdV case, up to the Ehrenfest time $O(|\log \ve |\ve^{-1})$, provided the dynamical laws are well defined. In their model, the perturbation is of linear type, which do not allow large variations of the soliton shape, different to the scaling itself. 

\smallskip

In \cite{Mu2,Mu3} it was described the soliton dynamics, for all time, in the case of the time independent, perturbed gKdV equation (\ref{aKdV0}). The main novelty was the understanding of the dynamics as a nonlinear interaction between the soliton and the potential, in the spirit of the recent works of Holmer-Zworski \cite{HZ}, and Martel-Merle \cite{MMcol1,MMcol2}. In order to state this last result, and our present main results, let us first describe the framework that we have considered for the potential $a(\cdot)$ in (\ref{aKdV0}). 

\subsection*{Setting and hypotheses on $a(\cdot)$} Concerning the function $a$ in (\ref{aKdV0}), we assume that $a\in C^3(\R)$ and there exist fixed constants $K, \ga>0$ such that
\be\label{ahyp} 
\begin{cases}
1< a(r) < 2, \quad a'(r)>0, \quad \hbox{ for all } r\in \R, \\
0<a(r) -1 \leq  Ke^{\ga r}, \; \hbox{ for all } r\leq 0,  \qquad 0<2-a(r)\leq K e^{-\ga r} \; \hbox{ for all } r\geq 0,\, \hbox{ and}\\
 | a^{(k)}(r)| \leq K e^{-\ga|r|},  \quad \hbox{ for all } r\in \R,  \; k=1,2,3.
\end{cases}
\ee
In particular, $\lim_{r\to -\infty}a(r) = 1$ and $\lim_{r\to +\infty} a(r) = 2$. The chosen limits (1 and 2) do not imply a loss of generality, it just simplifies the computations. In addition, we assume the following hypothesis: there exists $K>0$ such that for $m=2,3$ and $4$,
\be\label{3d1d}
| (a^{1/m})^{(3)}(s) | \leq K (a^{1/m})'(s), \quad \hbox{ for all } \quad s\in \R.
\ee
This condition is generally satisfied, however $a'(\cdot)$ cannot be a compactly supported function. 

\smallskip

We remark some important facts about (\ref{aKdV0}) (see \cite{Mu2,Mu3} for more details). Firstly, this equation is not  invariant under scaling and spatial translations. Second, a nonzero solution of (\ref{aKdV0}) \emph{might lose or gain some mass}, depending on the sign of $u$, in the sense that, at least formally, the Mass
\be\label{Ma}
M[u](t):= \frac 12\int_\R u^2(t,x)\,dx    \quad \hbox{ satisfies} \quad  \partial_t M[u](t) = -\frac{\ve}{m+1} \int_\R a'(\ve x) u^{m+1}(t,x)dx.
\ee
On the other hand, the energy
\be\label{Ea}
E_a [u](t) :=  \frac 12 \int_\R u_x^2(t,x)\,dx + \frac \la 2 \int_\R u^2(t,x)\, dx - \frac 1{m+1}\int_\R a(\ve x)  u^{m+1}(t,x)\,dx
\ee
remains formally constant for all time. Let us recall that this quantity is conserved for local $H^1$-solutions of (\ref{gKdV}).  
Since $a\sim 1$ as $x\to -\infty$, one should be able to construct a generalized soliton-like solution $u(t)$, satisfying $u(t) \sim Q(\cdot -(1-\la)t)$ as $t\to -\infty$.\footnote{Note that, with no loss of generality, we have chosen the scaling parameter equals one.} Indeed, this {\bf scattering} result has been proved in \cite{Mu2}, but for the sake of completeness, it is briefly described in the following paragraph.

\subsection*{Description of the dynamics} Let us recall the setting of our problem. Let $0<\la< 1$ be a fixed parameter, consider the equation
\be\label{aKdV}
\begin{cases}
u_t + (u_{xx} -\la u + a (\ve x) u^m)_x =0 \quad \hbox{ in \ } \R_t \times \R_x,   \\
m=2,3 \hbox{ and } 4;\quad  0< \ve\leq\ve_0;  \quad a(\ve \cdot) \hbox{ satisfying } (\ref{ahyp}) \hbox{-}(\ref{3d1d}). 
\end{cases}
\ee
Here $\ve_0>0$ is  a small parameter. Let $\la_0$ be the following parameter:
\be\label{l0}
\la_0:= \frac{5-m}{m+3} \in(0,1).
\ee
Assuming the validity of (\ref{aKdV}), one has the following generalization of \cite{Martel}:

\begin{thm}[Existence of solitons for gKdV under variable medium, \cite{Mu2}]\label{MT} Suppose $m=2,3$ and $4$. Let $0\leq \la <1$ be a fixed number.
There exists a small constant $\ve_0>0$ such that for all $0<\ve<\ve_0$ the following holds. There exists a solution $u\in C(\R, H^1(\R))$ of (\ref{aKdV0}), global in time, such that 
\be\label{Minfty}
\lim_{t\to -\infty} \|u(t) - Q(\cdot -(1-\la)t) \|_{H^1(\R)} =0.
\ee
\end{thm}

Let us remark that (\ref{Minfty}) can be improved in the following way: there exists $K,\ga>0$ such that
\be\label{expode}
\|u(t)-  Q(\cdot -(1-\la)t) \|_{H^1(\R)} \leq Ke^{\ga\ve t}, \quad \hbox{for all $t\lesssim \ve^{-1-1/100}$\; (cf. \cite{Mu2}).}
\ee

\smallskip

Next, we have described the dynamics of interaction soliton-potential. Let $\la\in (0,1)$, and let $\tilde \la  =\tilde \la(m)$ be the unique solution of the algebraic equation \cite{Mu3}
\be\label{tlan}
\tilde \la (\frac {1-\la_0}{\tilde \la -\la_0})^{1-\la_0} =2^{\frac 4{m+3}}, \quad \la_0<\tilde \la<1,
\ee
with $\la_0$ given by (\ref{l0}). Let  $\kappa(\la)$ be the parameter defined by
\be\label{kmla}
\kappa(\la) := 2^{-1/(m-1)}, \quad 0<\la<\tilde \la; \qquad \kappa(\la) =1, \quad \tilde \la<\la<1. 
\ee
The above numbers represent a sort of \emph{equilibria} between the energy of the solitary wave and the strength of the potential. Indeed, let $c_\infty =c_\infty(\la)$ be the unique solution of the algebraic equations \cite{Mu2,Mu3}
\be\label{cinfla}
c_\infty^{\la_0}\Big(\frac{\la -c_\infty \la_0}{\la-\la_0}\Big)^{1-\la_0} = \begin{cases} 2^{4/(m+3)},& c_\infty >\la, \quad 0<\la<\tilde \la, \\
1,& c_\infty  <\la, \quad \tilde \la<\la<1,\end{cases}
\ee
and $c_\infty(\la_0) =1$, respectively.  We claim that this number represents the \emph{final scaling} of the soliton. Indeed, one has  $c_\infty(\la)\geq 1$ if $0<\la\leq \la_0$, $\la <c_\infty(\la) <1$  if $\la_0<\la<\tilde \la$, $0<c_\infty(\la)<\la$ if $\tilde \la<\la<1$, and the following

\begin{thm}[Interaction soliton-potential: refraction and reflection, \cite{Mu2,Mu3}]\label{MTL1} Suppose $0< \la <1$, with $\la\neq \tilde \la$. There exists $K,\ve_0>0$ such that for all $0<\ve<\ve_0$ the following holds. There exists constants $ \tilde T, c^+>0$, and a smooth $C^1$ function $ \rho(t)=\rho_\la(t) \in \R $ such that the function $w^+(t) := u(t) - \kappa(\la) Q_{c^+} (\cdot-  \rho(t))$ satisfies, for all $t\gg \ve^{-1}$,
\be\label{St1l}
\| w^+(t) \|_{H^1(\R)} + | \rho'(t) - c_\infty(\la) +\la | +|c^+-c_\infty| \leq K\ve^{1/2}.
\ee
\end{thm}
\begin{rem}[The limiting case $\la=\tilde \la$]
The behavior of the solution in the case $\la =\tilde \la $ remains an interesting open problem. 
\end{rem}

\begin{rem}
In addition to (\ref{St1l}), it is proved in \cite{Mu2,Mu3} an asymptotic stability property, in the spirit of Martel and Merle \cite{MMnon}. This result gives the existence of the limiting parameter $c^+$ above mentioned. We believe that  the above is the first mathematical proof of the existence of a global, \emph{reflected} soliton-like solution in a variable coefficients gKdV model.
\end{rem}

Finally, by means of a contradiction argument, no pure soliton solutions are present in this  regime.

\begin{thm}[Non existence of pure-soliton solutions for (\ref{aKdV}), \cite{Mu2,Mu3}]\label{MTL2} Let $0<\la<1$, with $\la\neq \tilde \la$. Then 
$$
\limsup_{t\to +\infty} \|w^+(t)\|_{H^1(\R)}>0.
$$
\end{thm}

\subsection*{Main result} A natural question left open in \cite{Mu2,Mu3} is to establish a quantitative lower bound on the defect $w^+(t)$ as the time goes to infinity, at least in the case $0<\la<1$, $\la\neq \tilde \la$ (the cases $\la=0$ and $\la =\tilde \la$ seem harder). In this paper we improve Theorem \ref{MTL2} by showing a first lower bound on the defect $w^+(t)$ at infinity. In other words, any perturbation of the constant coefficients gKdV equation of the form (\ref{aKdV}) induces non trivial dispersive effects on the soliton solution --they are not pure anymore. This result clarifies the existence of a nontrivial {\bf dispersive tail} and the {\bf inelastic character} of generalized solitons for perturbations of some dispersive equations, and moreover, it seems to be the general behavior. In addition, one can see this result as a generalization to the case of interaction soliton-potential of the recent ones proved by Martel and Merle, concerning the inelastic character of the collision of two solitons for non-integrable gKdV equations \cite{MMcol1,MMfin2}.

\smallskip

However, in order to obtain such a quantitative bound, and compared with the proof in \cite{MMfin2}, the present analysis requires several new ideas, in particular for the more difficult case, the cubic one. As we will describe below, our lower bounds are related to first and second order corrections to the dynamical parameters of the soliton solution. The main result of this paper is the following

\begin{thm}[Inelastic character of the soliton-potential interaction]\label{MTL} Let $m=2,3$ and $4$, $0<\la<1$, $\la\neq \tilde \la$, and $\delta := \frac 1{50}$. There exist constants $K,\ve_0>0$ such that, for all $0<\ve<\ve_0$, the following holds. Let $w^+(t)$ be as in Theorem \ref{MTL1}. Then 
\be\label{LBound}
\liminf_{t\to +\infty} \|w^+(t)\|_{H^1(\R)} \geq \frac 1K \ve^{1 +\delta},
\ee
\end{thm}

\begin{rem}[Meaning of $\delta$]
The number $\delta$ above is needed in our computations, but it is not essential. It is related to the definition of the time of interaction $T_\ve$ (\ref{Te}) and estimates (\ref{expode}) and (\ref{cTee}), but it can be replaced by any $\delta>0$ provided $\ve_0$ is chosen even smaller. Looking at our proofs, we believe that the best lower bound is given by $\sim\ve |\log\ve|^{-\delta}$, for some $\delta>0$; however, this problem will not be considered in this paper.
\end{rem}

\begin{rem}
Similar to the results obtained in \cite{MMcol1,MMfin2}, we have found a nontrivial gap between the two bounds (\ref{St1l}) and (\ref{LBound}). This gap is related to the emergence of infinite mass corrections to the constructed approximate solution \cite{Mu2,Mu3}, and it is not formally present in the NLS model \cite{Mu1}. The understanding of this gap is a very interesting open problem. Additionally, from the above results we do not discard the existence of small solitary waves (note that small solitons move to the left), at least for the case $m=2$. In the cubic and quartic cases, we believe there are no such soliton solutions. 
\end{rem}

\subsection*{Ideas of the proof}  As we have explained before, the above result is originally based in a recent argument introduced by Martel and Merle in \cite{MMfin2}, to deal with the interaction of two nearly-equal solitons for the quartic gKdV equation. Roughly speaking, in their paper the interaction was proved to be inelastic because of a small lack of symmetry on the soliton trajectories, contrary to the symmetric integrable case. In this paper, we improve the Martel-Merle idea in two directions: first, we generalize such an argument to the case of the interaction soliton-potential, which is nontrivial since our problem has no evident symmetries to be exploited; and second, we deal, in addition, with a somehow \emph{degenerate} case, the cubic one, where the original Martel-Merle argument is not longer available. Therefore, we introduce new ideas to recover the same bound as in the other cases.

\smallskip

Let us describe the proof. We consider an approximate solution of (\ref{aKdV}), describing the interaction soliton-potential. This problem was first considered in \cite{Mu2}, but in order to find an explicit expression for the defect of the solution, we improved such a construction in \cite{Mu3}. 

\smallskip

Let us be more precise. The objective of the new approximate solution is to obtain \emph{first and second order corrections} on the translation and scaling parameters $\rho(t), c(t)$ of the soliton solution. Indeed, in \cite{Mu3} was proved that the solution $u(t)$ behaves along the interaction, at first order, as follows:
$$
u(t,x) \sim a^{-1/(m-1)}(\ve \rho(t)) Q_{c(t)}(x-\rho(t)) + \hbox{ lower order terms in $\ve$,} 
$$
with $(c,\rho)$ satisfying the dynamical laws\footnote{We write $f_j=f_j(\ve t)$ in order to emphasize the fact that we are working with slowly varying functions, but in the rigorous proof below we only use the notation $f_j(t).$ }
\bea\label{rhoplus}
 c'(t) & \sim & \ve f_1(\ve t) + \ve^2 f_3(\ve t), \; \hbox{ with } \;  f_1(\ve t) , f_3(\ve t) \neq 0, \quad m=2,3,4; \\
 \rho'(t) & \sim & c(t) -\la + \ve f_2(\ve t) +\ve^2 f_4(\ve t) , \; \hbox{ with } \;  f_2(\ve t) \neq 0, \quad m=2,4,\label{cplus}
\eea
(see Proposition \ref{prop:decomp1} for an explicit description of this dynamical system). Moreover, one has $f_2 \equiv 0$ in the cubic case (cf. Proposition \ref{prop:decomp}). Roughly speaking, the parameter $f_2(\ve t)$ ($f_3(\ve t)$ resp.) satisfies 
$$
\int_\R  \ve f_2 (\ve t)dt <+\infty  \qquad ( \int_\R \ve f_3(\ve t)dt <+\infty \hbox{ resp.}).
$$ 
Therefore, after integration in a time interval of size $O(\ve^{-1})$, near $t\sim 0$, these new terms induce a correction of order $O(1)$ (of order $O(\ve)$ resp.) on the trajectory $\rho(t)$ (on the scaling $c(t)$, resp.). These corrections are precisely the quantities that induce lower bounds for the hidden defect.

\smallskip

The next step is to introduce a new function, say $v(t)$, which has the opposite behavior compared to $u(t)$. This solution is pure as $t\to +\infty$, and therefore, from Theorem \ref{MTL1}, different from $u(t)$. We can describe the dynamics associated to $v(t)$ for all time, in particular along the interaction region: we construct an approximate solution $\tilde v(t)$, with associated dynamical parameters $\tilde c(t)$ and $\tilde\rho(t)$, which satisfy suitable dynamical laws, as in (\ref{rhoplus})-(\ref{cplus}). However, since $v(t)$ is pure as $t\to +\infty$, the respective coefficients $\tilde f_3(t)$ and $\tilde f_2(t)$ are of {\bf different signs} with respect to  (\ref{rhoplus})-(\ref{cplus}). This crucial observation was first noticed by Martel and Merle in \cite{MMfin2} for the quartic gKdV model, and represents a lack of symmetry in the dynamics.

\smallskip

The purpose for the rest of proof is to profit of this fact. The idea is the following: if (\ref{LBound}) is not satisfied, then $u(t)$ and $v(t)$ are \emph{very close} for all time, at some order {\bf smaller} than $\ve$. This property is nothing but a backward stability result.\footnote{The existence of this property in the NLS case is an open problem, see \cite{Mu1}.} Now, suppose for instance that we are in the quadratic case. From the above stability result, one can prove that the dynamical parameters of $u(t)$ and $v(t)$  are very close, in the sense that
\be\label{lila}
|c(t) -\tilde c(t)| \ll \ve, \quad |\rho(t) -\tilde \rho(t)| \ll 1.
\ee
We give a more precise description of these properties in Lemmas \ref{BaSta} and \ref{6p2}, and (\ref{LaLb}). But from (\ref{rhoplus})-(\ref{cplus}) one has, after integration in time,
\be\label{lila2}
|\rho(t) -\tilde \rho(t)| \sim \int_\R f_2(s)ds .
\ee
Note that we have used that $f_2(\ve t)$ and $\tilde f_2(\ve t)$ have opposite signs. Then we have a contradiction with (\ref{lila}), provided the integral in (\ref{lila2}) is not zero, and the bounds in (\ref{lila}) are small enough.

\smallskip

The last step above can be performed in a more rigorous way with the following argument. In the case $m=2,4$ the idea is to find a quantity satisfying the following properties: $(i)$ its variation in time is of order $O(\ve)$, $(ii)$ it contains the dynamical laws (\ref{rhoplus})-(\ref{cplus}), and $(iii)$ now the term $\ve f_2(t)$ is relevant for the dynamics. This quantity is given by a suitable modification of a well-known functional $J(t)$ introduced by Martel and Merle in \cite{MMfin2}, whereas in the cubic case the defect is in some sense degenerate and therefore $J(t)$ is useless. However, since in this case the variation of $c(t)$ is of second order in $\ve$, we still recover the same lower bound, but we require several sharp estimates. We overcome this difficulty by using improved Virial estimates (cf. Lemmas \ref{VL}, \ref{VLtil}), with the right signs, which allow to close our arguments.  To obtain a suitable lower bound for the defect in the case $\la=0$ is probably a more challenging, open problem.

\begin{rem}[The Schr\"odinger case]
The interaction soliton-potential has be also considered in the case of the nonlinear Schr\"odinger equation with a slowly varying potential, or a soliton-defect interaction. See e.g. Gustafson et al. \cite{GFJS, FGJS}, Gang and Sigal \cite{GS}, Gang and Weinstein \cite{GW}, Holmer, Marzuola and Zworski \cite{HZ,HMZ0, HMZ}, Perelman \cite{P} and our recent work \cite{Mu1} on the NLS equation. It is relevant to say that the equivalent of Theorem \ref{MTL} for the Schr\"odinger case considered in \cite{Mu1} is an interesting open question.
\end{rem}

\smallskip

Let us explain the organization of this paper. First, in Section \ref{2} we introduce the basic tools for the study of the interaction problem. These results are reminiscent of our previous papers \cite{Mu2,Mu3}, and therefore are stated without proofs.  In Section \ref{4} we consider the case of a decreasing potential. We introduce the solution $v(t)$ which satisfies the opposite behavior with respect to $u(t)$. Section \ref{5} is devoted to the rigorous proof of (\ref{lila}), and in Section \ref{6} we prove (\ref{lila2}). In Section \ref{7} we prove the main result in the cases $m=2,4$, and finally in Section \ref{8} we consider the most difficult case, $m=3$.

\smallskip

\noindent
{\bf Notation.} We follow most of the notation introduced in \cite{Mu3}. In particular, in this paper both $K,\ga>0$ will denote fixed constants, independent of $\ve$, and possibly changing from one line to another. 
Additionally, we introduce, for $\ve>0$ small, the time of interaction
\be\label{Te}
 T_\ve := \frac {\ve^{-1 -\frac 1{100}}}{1-\la}>0.
\ee

\noindent
{\bf Acknowledgments}. I wish to thank Y. Martel and F. Merle for their continuous encouragement along the elaboration of this work. Part of this work has been written at the University of Bilbao, Spain. The author has been partially funded by grants Anillo ACT 125 CAPDE and Fondo Basal CMM. Some of these results have been announced in \cite{Mu5}.

\medskip

\section{Preliminaries}\label{2}

The purpose of this section is to recall several properties needed along this paper. For more details and the proofs of these results, see Section 2 and 3 in \cite{Mu2,Mu3}.

\subsection{Existence of approximate parameters} 

Denote, for $C>0$, $P\in \R$ given, and $m=2,3$ or $4$,
\be\label{f1}
f_1(C,P) := \frac 4{m+3} \ C(C-\frac \la{\la_0} ) \frac{a'(\ve P)}{a(\ve P)}.
\ee
We recall the existence of a unique solution for a dynamical system involving the evolution of the first order \emph{scaling} and \emph{translation} parameters of the soliton solution, $(C(t), P(t))$, in the interaction region. The behavior of this solution is essential to understand the dynamics of the soliton inside this region.

\begin{lem}[\cite{Mu2,Mu3}]\label{ODE} Let $m=2,3$ or $4$. Let $\la_0, a(\cdot)$ and $f_1$ be as in (\ref{l0}), (\ref{ahyp}) and (\ref{f1}). There exists $\ve_0>0$ small such that, for all $0<\ve<\ve_0$, the following holds.

\begin{enumerate}
\item \emph{Existence}. 
Consider $0\leq \la<1$. There exists a unique solution $(C(t),P(t))$, with $C(t)$ bounded, monotone and positive, defined for all $t\geq -T_\ve$, of the following nonlinear system  
\be\label{c}
\begin{cases}
C'(t) = \ve f_1(C(t), P(t)), & C(- T_\ve) = 1, \\
P'(t) = C(t) -\la, & P(-T_\ve) =-(1-\la)T_\ve.
\end{cases}
\ee
Moreover, $\lim_{t\to +\infty} C(t) >0$, for all $0\leq \la<1$, independently of $\ve$.

\item \emph{Asymptotic behavior}. Let $\la_0<\tilde \la<1$ be the unique number satisfying (\ref{tlan}).
Then,
\begin{enumerate}
\item For all $0\leq \la  \leq  \tilde \la$, one has  $\lim_{t\to +\infty} C(t) >\la$ and $\lim_{t\to +\infty} P(t) = +\infty$.

\item For all $\tilde \la<\la<1$, there exists a unique $ t_0\in ( -T_\ve, +\infty)$ such that $C(t_0)=\la$, and  $\lim_{t\to +\infty} C(t) < \la $. Moreover, $\lim_{t\to +\infty} P(t) =-\infty.$ Finally, one has the bound  $-T_\ve <t_0 \leq K(\la) T_\ve ,$ for a positive constant $K(\la)$, independent of $\ve$.
\end{enumerate}
\end{enumerate}
\end{lem}

\begin{rem}
From the above result, one can define a unique {\bf time of escape} $\tilde T_\ve >-T_\ve$ such that $P(t)$ satisfies 
\be\label{timeescape}
P(\tilde T_\ve) = (1-\la)T_\ve,  \hbox{ for } \ 0<\la<\tilde \la, \qquad P(\tilde T_\ve) = -(1-\la)T_\ve,  \hbox{ for }\ \tilde \la<\la<1. 
\ee
(See \cite[Definition 3.1]{Mu3}.) In addition, 
\be\label{tte}
\tilde T_\ve \leq K(\la) T_\ve, \quad 0<K(\la)<+\infty,
\ee
provided $\la\neq \tilde \la$. Moreover, one has $C(\tilde T_\ve) = c_\infty(\la) + O(\ve^{10}), $
with $c_\infty(\la)$ being the unique solution of the algebraic equation (\ref{cinfla}). See \cite{Mu2,Mu3} for the proof of these results.
\end{rem}

\subsection{Construction of an approximate solution describing the interaction \cite{Mu3}}\label{sec:2}

 Let $t\in [-T_\ve, \tilde T_\ve]$, $Q_c$ given in (\ref{(3)}), $c=c(t)>0$ and $\rho(t)\in \R$ be bounded functions to be chosen later, and
\be\label{defALPHA}
    y:=x-\rho(t), \quad     R(t,x): =\tilde a^{-1}(\ve \rho(t)) Q_{c(t)}(y),
\ee
where $\tilde a (s) := a^{\frac 1{m-1}}(s).$ The parameter $\tilde a$ describes the shape variation of the soliton along the interaction. Concerning the parameters $c(t)$ and $\rho(t)$, it is assumed that, for all $t\in [-T_\ve, \tilde T_\ve]$,
\be\label{r1}
|c(t)-C(t)| +  |\rho'(t) -P'(t) |\leq \ve^{1/100}.
\ee
with $(C(t),P(t))$ from Lemma \ref{ODE}.  Consider a cut-off function $\eta \in C^\infty (\R)$ satisfying
$$
0\leq \eta (s) \leq 1, \quad 0\leq  \eta' (s) \leq 1, \; \hbox{ for any } s\in \R; \quad
\eta(s)\equiv 0 \; \hbox{ for } s\leq -1, \quad  \eta(s)\equiv 1 \; \hbox{ for } s\geq 1.
$$
Define 
\be\label{etac}
\eta_\ve (y) := \eta( \ve y +  2 ),
\ee
From \cite{Mu3}, the form of $\tilde u(t,x)$, the approximate solution, will be the sum of a soliton plus a correction term:
\be\label{defv} 
\tilde u(t,x) :=\eta_\ve (y)(R(t,x)+w(t,x)),
\ee
where $w$ is given by
\be\label{defW}
    w(t,x):= \begin{cases} \ve d(t)A_{c} (y) , \quad \hbox{ if $m=2,4$}, \\  \ve d(t)A_{c} (y)  + \ve^2B_c(t,y), \quad \hbox{ if $m=3$}, \end{cases}
\ee
and  $d(t) := (a'\tilde a^{-m}) (\ve \rho(t)).$ Here $A_{c}(y)$ and   $B_c(t,y)$ are unknown functions.  Note that, by definition, $\tilde u(t, x) = 0$ for all  $y\leq - 3\ve^{-1}$.

We want to estimate the size of the error obtained by inserting $\tilde u$ as defined in (\ref{defv})-(\ref{defW}) in the equation (\ref{aKdV}). For this, we define the residual term
\be\label{2.2bis}
S[\tilde u](t,x) := \tilde u_t + (\tilde u_{xx} -\la \tilde u +a(\ve x) \tilde u^{m})_x.
\ee
For this quantity one has the following 

\begin{prop}[\cite{Mu2,Mu3}]\label{prop:decomp} Suppose $(c(t), \rho(t))$ satisfying (\ref{r1}). There exists $\ga>0$ independent of $\ve$ small, and an approximate solution $\tilde u$ of the form (\ref{defv})-(\ref{defW}), such that for all $t\in [-T_\ve, \tilde T_\ve]$, one has:

\begin{enumerate}

\item \emph{Almost solution}. The error associated to the function $\tilde u(t)$ satisfies
\begin{align*}
 S[\tilde u]  &  =   (c'(t) - \ve f_1(t) -\ve^2 \delta_{m,3} f_3(t))\partial_c\tilde u   \\
 &  + (\rho'(t) -c(t)+ \la -  \ve f_2(t) -\ve^2 \delta_{m,3}f_4(t)) \partial_\rho  \tilde u + \tilde S[\tilde u](t),
\end{align*}
with $\delta_{m,3}$ the Kronecker symbol, $ \partial_\rho  \tilde u :=  \partial_\rho R + O_{H^1(\R)}(\ve^{1/2} e^{-\ve\ga|\rho(t)|})$,  and
\be\label{SH2}
\| \tilde S[\tilde u](t) \|_{H^1(\R)} \leq K \ve^{3/2}e^{-\ga \ve |\rho(t)|}.
\ee

\item $A_c, B_c$ satisfy 
\be\label{Ac}
A_c, \partial_c A_c \in L^\infty(\R), \quad A_c'\in \mathcal Y,  \quad
|A_c(y) | \leq K e^{-\ga y} \; \hbox{ as } y\to +\infty, \quad \lim_{-\infty} A_c  \neq 0,
\ee
and for $m=3$,
\be\label{Bc}
\begin{cases}
B_c'(t,\cdot) \in L^\infty(\R),\quad  |B_c(t, y) | \leq K e^{-\ga y}e^{-\ve \ga|\rho(t)|} \; \hbox{ as } y\to +\infty, \\
|B_c(t,y)| +|\partial_c B_c(t,y)| \leq K|y| e^{-\ve \ga|\rho(t)|}, \; \hbox{ as } y\to -\infty, 
\end{cases}
\ee

\item  \emph{$L^2$-solution}. For all $t\in [-T_\ve, \tilde T_\ve]$, $\eta_\ve w(t, \cdot ) \in H^1(\R)$, with
\be\label{H1}
\| \eta_\ve w(t, \cdot ) \|_{H^1(\R)} \leq K \ve^{1/2}e^{-\ga \ve |\rho(t)|},
\ee
and
\be\label{AO}
\abs{\int_\R \eta_\ve w(t,x)Q_c(y)dx} + \abs{\int_\R y \eta_\ve w(t,x)Q_c(y)dx} \leq K \ve^{10}.
\ee

\item In addition, $f_1 (t)=f_1(c(t),\rho(t))$ is given by (\ref{f1}),
\be\label{f2}
 f_2(t) = f_2(c(t), \rho(t))  :=  - \frac{\xi_m}{\sqrt{c(t)}} (\la - 3 \la_0c(t)) \frac{a'}{a}(\ve \rho(t)), \quad  \xi_m   :=  \frac{(3-m)}{(5-m)^2} \frac{(\int_\R Q)^2}{\int_\R Q^2},
\ee
\be\label{f3}
f_3(t) = f_3(c(t), \rho(t)) :=  \frac{\tilde\xi_3}{\sqrt{c(t)}} (c(t)-\la) \frac{a'^2}{a^2}(\ve \rho(t)), \quad   \tilde \xi_3 := \frac \la2 \frac{(\int_\R Q)^2}{\int_\R Q^2},
\ee
and $f_4(t)$ satisfies the decomposition
\be\label{f4}
f_4(t) := f_4^1(t)\frac{a'^2}{a^2}(\ve \rho(t)) + f_4^2(t) \frac{a''}{a}(\ve \rho(t)), \quad  |f_4^i (t)| \leq K.
\ee

\item Finally, one has the estimates
\be\label{SIn}
\abs{\int_\R Q_c \tilde S[\tilde u]} +\abs{\int_\R yQ_c \tilde S[\tilde u]} \leq K \ve^2 e^{-\ve\ga|\rho(t)|} + K\ve^3,
\ee
for $m=2,4$, and 
\be\label{SIn3}
\abs{\int_\R Q_c \tilde S[\tilde u]} +\abs{\int_\R yQ_c \tilde S[\tilde u]} \leq K \ve^3 e^{-\ve\ga|\rho(t)|} + K\ve^4,
\ee
in the case $m=3$.

\end{enumerate}
\end{prop}

\begin{rem}
Note that, even under a correction term of second order, namely $\ve^2 B_c$, one cannot improve the associated error (\ref{SH2}). We believe that this phenomenon is a consequence of the fact that $A_c \not\in L^2(\R)$. 
\end{rem}

\subsection{Decomposition of the solution in the interaction region}\label{sec:3}

The next result summarizes the interaction soliton-potential. Roughly speaking, the solution $u(t)$ behaves as the approximate solution $\tilde u(t)$.

\begin{prop}[\cite{Mu3}]\label{prop:I} Suppose $\la\in (0,1)$, $\la\neq \tilde \la$. There exist $K_0,\ve_0>0$ such that the following holds for any $0<\ve <\ve_0$. 

\ben
\item
There exist unique $C^1$ functions $c, \rho : [-T_\ve, \tilde T_\ve] \to \R$ such that, for all $t\in [-T_\ve, \tilde T_\ve]$, the function $z(t) := u(t)-\tilde u(t,c(t), \rho(t))$ satisfies
\be\label{INT41}
 \| z(t) \|_{H^1(\R)} \leq K_0 \ve^{1/2},Ê\quad \int_\R z(t,x) y Q_c(y) dx  = \int_\R z(t,x) Q_c(y)dx=0.
\ee
In addition, $z(t)$ solves the following gKdV equation
\be\label{Eqz1}
\begin{cases}
z_t +  \big\{ z_{xx}  -\la z +  a(\ve x) [ (\tilde u +z)^m - \tilde u^m ] \big\}_x  +  \tilde S[\tilde u]     + c_1'(t) \partial_c \tilde u  + \rho_1'(t) \partial_\rho \tilde u =0,\\
 c_1' (t):= c'(t)-\ve f_1(t)-\ve^2 \delta_{m,3}f_3(t), \; \ \rho_1'(t) := \rho'(t) -c(t)+\la- \ve f_2(t)-\ve^2 \delta_{m,3} f_4(t).
\end{cases}
\ee

\item There is $\ga>0$ independent of $K_0$ such that for every $t\in [-T_\ve, T^*]$,
\be\label{rho1}
 |\rho_1'(t) |   \leq    K  (m-3 +\ve e^{-\ga\ve|\rho(t)|} ) \Big[\int_\R z^2 e^{- \ga\sqrt{c}|y|} \Big]^{1/2}   + K \int_\R e^{- \ga\sqrt{c}|y|}z^2(t) + K\abs{\int_\R yQ_c \tilde S[\tilde u]},
\ee
\be\label{c1}
|c_1'(t)  | \leq  K  \int_\R e^{-\ga\sqrt{c} |y|} z^2(t)  +  K\ve e^{-\ga\ve|\rho(t)| } \Big[ \int_\R e^{-\ga\sqrt{c} |y|} z^2(t)\Big]^{1/2} + K\abs{\int_\R Q_c \tilde S[\tilde u]},
\ee
and
\be\label{todo}
|c(t)-C(t)| +|\rho'(t)-P'(t)| \leq K\ve^{1/2}.
\ee
Finally,
\be\label{cTee}
|c(-T_\ve) -C(-T_\ve)|+ |\rho(-T_\ve) -P(-T_\ve) | + \|z(-T_\ve)\|_{H^1(\R)} \leq K\ve^{10},
\ee
with $K>0$ independent of $K_0$.
\een
\end{prop}

\begin{rem}
Note that estimates (\ref{todo}) improve (\ref{r1}). In addition, (\ref{cTee}) are consequences of  (\ref{expode}) at time $-T_\ve$, and (\ref{c}). Moreover, from the proof of the above result, (\ref{expode}) and (\ref{SH2}), one can see that e.g. an estimate of the order $\|z(t)\|_{H^1(\R)} \leq K_0\ve^{10}$ is valid for all sufficiently early times, namely $t\leq -K\ve^{-1}|\log\ve|$, with $K>0$ large enough. 
\end{rem}

\subsection{Virial estimate}
A better understanding of the estimate on the scaling parameter (\ref{c1}) needs the introduction of a Virial estimate, in the spirit of \cite{Mu2} (see Lemma 6.4). See also \cite{H} for a similar result.

First, we define some auxiliary functions. Let $\phi \in C^\infty(\R)$ be an \emph{even} function satisfying the following properties
\be\label{psip}
\begin{cases}
\phi' \leq 0 \; \hbox{ in } [0, +\infty); \quad  \phi  \equiv 1 \; \hbox{ in } [0,1], \\
\phi (x) = e^{-x}  \; \hbox{ on } [2, +\infty) \quad\hbox{and}\quad  e^{-x} \leq \phi (x) \leq 3e^{-x}  \; \hbox{ on } [0,+\infty).
\end{cases}
\ee
Now, set $\psi(x) := \int_0^x \phi $. It is clear that $\psi$ is an odd function. Finally, for $A>0$, denote 
\be\label{psiA}
\psi_A(x) := A(\psi(+\infty) + \psi(\frac xA))>0. 
\ee
Note that $\lim_{x\to -\infty} \psi(x) =0$ and  $e^{-|x|/A} \leq \psi_A'(x)   \leq 3e^{-|x|/A}$. We claim the following

\begin{lem}[\cite{Mu3}]\label{VL} There exist $K, A_0, \delta_0>0$  such that for all $t\in [-T_\ve, \tilde T_\ve]$ and for some $\ga =\ga(A_0)>0$,
\be\label{dereta}
 \partial_t \int_\R  z^2(t,x) \psi_{A_0}(y)   \leq   -\delta_0  \int_\R ( z_x^2 + z^2 )(t,x) e^{-\frac 1{A_0} |y|}  + KA_0 \ve^{5/2}.
\ee
\end{lem}

A simple but very important conclusion of the last estimate, is the following: one has, from (\ref{cTee}) and (\ref{dereta}),
\bea\label{intc1}
\int_{-T_\ve}^{t} \int_\R ( z_x^2 + z^2 )(t,x) e^{-\frac 1{A_0} |y|} dx ds  & \leq &  K \Big[\int_\R  z^2(-T_\ve) \psi_{A_0}(y) - \int_\R  z^2(t) \psi_{A_0}(y)\Big]  + K \ve^{3/2-1/100} \nonu \\
& \leq & K \ve^{3/2- 1/100}.
\eea
for all $t\in [-T_\ve, \tilde T_\ve]$, by taking $A_0$ large enough, independent of $\ve$ and $K^*$. In other words, we improve the estimate on the integral of $z^2+ z_x^2$ near the soliton (a crude integration of (\ref{intc1}) gives a bound $ K\ve^{-\frac 1{100}}$). Finally, from (\ref{c1}) and (\ref{SIn})-(\ref{SIn3}), we improve estimate (4.57) in \cite{Mu3}, to obtain
\be\label{Intec1}
\int_{-T_\ve}^{t} |c_1'(s)|ds \leq  K\ve^{3/2- 1/100}.
\ee
(See \cite[estimate (4.73)]{Mu3} for the integration of terms of the form $\ve e^{-\ga\ve|\rho(t)|}$.)

\medskip

\section{The case of a decreasing potential}\label{4}

In this section we deal with the problem of existence of a \emph{pure} soliton-like solution as time goes to $+\infty$. Our objective is to briefly describe the dynamics of such a solution, say $v(t)$, in order to \emph{compare} its behavior with the solution $u(t)$ described in Theorems \ref{MT} and \ref{MTL1}.  We sketch some of these results, being straightforward generalizations of the results of Section 4 in \cite{Mu3}. First, we state the following existence result (see also \cite[Proposition 7.2]{Mu2}).

\begin{prop}\label{Existv}  Suppose $x_0\in \R$ and $0< \la <1$ fixed, with $\la\neq \tilde \la$. Let $c^+>0$ with $|c^+-c_\infty(\la)|\leq K\ve^{1/2}$, where $c_\infty=c_\infty(\la)>0$ is the scaling given by Theorem \ref{MTL1}. Let $\kappa(\la)$ be the parameter defined in (\ref{kmla}). For $\ve_0>0$ small enough, the following holds for any $0<\ve < \ve_0$. There exists a \emph{unique} solution $v \in C(\R, H^1(\R))$ of (\ref{aKdV}) such that 
\be\label{mlim0}
\lim_{t\to +\infty} \|v(t) - \kappa(\la)Q_{c^+}(\cdot - (c^+-\la)t -x_0) \|_{H^1(\R)} =0.
\ee
Moreover, there are constants $K,\ga>0$ such that
\be\label{mmTep}
\| v(t) -  \kappa(\la) Q_{c^+}(\cdot - (c^+-\la)t -x_0) \|_{H^1(\R)} \leq  K e^{-\ve\ga t},
\ee
provided $0<\ve<\ve_0$ small enough.
\end{prop}

\begin{rem}
This result has been proved in \cite{Mu2} for all $0<\la\leq \la_0$ (namely, with $\kappa(\la)=2^{-1/(m-1)}$). The key argument in the proof was the introduction of the modified mass $\mathcal M[u](t)$, given by
\be\label{Mback}
\mathcal M[u](t) :=\int_\R \frac {u^2(t,x)}{2 a(\ve x)} dx,
\ee
which satisfies \cite{Mu2}, for all $t, t' \geq \tilde T_\ve$, with $t'\geq t$, $\mathcal M[u](t) -\mathcal M[u](t') \leq K e^{-\ve \ga t}.$
From the proof of this result, we see that the same conclusion holds for any $\la_0<\la<\tilde \la$, with no differences in the proof, since one still has $c_\infty(\la)>\la$. However, in the case $\tilde \la<\la<1$, one has $c_\infty(\la) <\la$. Therefore, one needs a modification in the main argument of the proof. It turns out that, instead of considering the modified mass $\mathcal M[u]$, one should consider the modified mass $\hat M[u]$, introduced in \cite{Mu2}, given by
\be\label{hM}
\hat M[u](t) := \frac 12 \int_\R a^{1/m}(\ve x) u^2(t,x) dx.
\ee
Thanks to (\ref{3d1d}) this quantity satisfies, for any $m=2,3$ and $4$, the following property \cite{Mu2}: There exists $\ve_0>0$ such that, for all $0<\ve\leq \ve_0$, and for all $t'\geq t$,
\be\label{hM3}
\hat M[u](t) - \hat M[u](t') \geq 0.
\ee
After this modification, the proof of Proposition \ref{Existv} is direct from \cite[Proposition 7.2]{Mu2}.
\end{rem}

Let us come back to the study of the function $v(t)$. A straightforward consequence of Proposition \ref{Existv} is that, for all $\ve>0$ small enough,  
$$
\| v(\tilde T_\ve) - k(\la) Q_{c^+}(\cdot - (c^+ -\la)\tilde T_\ve -x_0) \|_{H^1(\R)} \leq  K \ve^{10}.
$$
Now, we want to describe the dynamics of this solution in the region $t\in [- T_\ve, \tilde T_\ve]$. The natural step is, following Section 4 in \cite{Mu3}, the \emph{construction of an approximate solution} $\tilde v(t)$, with dynamical parameters $\tilde c(t)$ and $\tilde \rho(t)$, of the form (compare with (\ref{defv}))
\be\label{hatv}
\tilde v(t) = \tilde v(t, \tilde c(t), \tilde \rho(t)) :=  \tilde{\eta}_\ve(\tilde y) (\tilde R(t) + \tilde w(t)),
\ee
such that $\tilde v(\tilde T_\ve)$ is close enough to $v(\tilde T_\ve)$. Here $\tilde y :=x-\tilde \rho(t)$, $\tilde{\eta}_\ve (\tilde y) := \eta (2-\ve \tilde y )$, $\tilde R(t)$ is the modulated soliton from (\ref{defALPHA}) with parameters $\tilde c(t)$ and $\tilde \rho(t)$, $\tilde d(t) := \frac{a'}{\tilde a^m}(\ve \tilde \rho(t))$, and
\be\label{tww}
\tilde w(t,x):= \begin{cases} 
\ve \tilde d(t)\tilde A_{\tilde c} (\tilde y) ,\quad \hbox{ if $m=2,4$}, \\
\ve \tilde d(t)\tilde A_{\tilde c} (\tilde y)  + \ve^2 \tilde B_{\tilde c}(t,\tilde y), \quad \hbox{ if $m=3$}. 
\end{cases}
\ee

\begin{rem}
Note that we have chosen $\tilde \eta_\ve$ such that $\tilde \eta_\ve(\tilde y) = 0$ for all $\tilde y\geq \frac 3\ve$, and $\tilde \eta_\ve(\tilde y) = 1$ for $\tilde y\leq \frac 1\ve$. This choice is the {\bf opposite} to the corresponding one associated to $\eta (y)$ (see (\ref{etac})). 
\end{rem}

Let $\la\in (0,1)$, $\la\neq \tilde \la$, and $X_0\in\R $ with $|X_0|\leq \ve^{-1/2-1/100}$. Let $(\tilde C(t), \tilde P(t))$ be the unique solution of the following \emph{backward} dynamical system (cf. Lemma \ref{ODE})
\be\label{ct}
\begin{cases}
\tilde C'(t)  = \ve f_1(\tilde C(t),\tilde P(t)), \qquad \tilde C(\tilde T_\ve) = c_\infty(\la), \\
\tilde P'(t) = \tilde C(t) -\la, \qquad \tilde P(\tilde T_\ve) = P(\tilde T_\ve) + X_0,
\end{cases}
\ee
(note that $\tilde P(\tilde T_\ve)$ can be negative, as in the case $\tilde \la<\la<1$.) For further purposes, we need the following 

\begin{lem}\label{Finn}
Assume $|X_0|\leq   \ve^{-1/2 -1/100}$. Let $(C(t),P(t))$ and $(\tilde C(t), \tilde P(t))$ be the solutions of (\ref{c}) and (\ref{ct}) respectively. Then, for all $t\in [-T_\ve, \tilde T_\ve]$,
\be\label{PtP}
\ve |P(t) -\tilde P(t)| + |C(t) -\tilde C(t)| \leq K\ve^{1/2-1/100}.
\ee
\end{lem}

\begin{proof}
We prove the most difficult case, namely $\la\in (\tilde \la, 1)$, since the case $\la\in(0,\tilde \la)$ is simpler. Suppose $t\in [-T_\ve, \tilde T_\ve]$, with $|t-t_0|\geq \frac \al\ve$, $t_0$ from Lemma \ref{ODE} and $\al>0$ a small number, independent of $\ve$. From \cite[identity (3.2)]{Mu3}, one has
$$
C^{\la_0}(t)(\frac \la{\la_0} -C(t))^{1-\la_0}  =  (\frac \la{\la_0}-1)^{1-\la_0} \frac{a^p(\ve P(t))}{a^p(\ve P(-T_\ve))} = (\frac \la{\la_0}-1)^{1-\la_0} a^p(\ve P(t)) (1+O(\ve^{10}) ). 
$$
Similarly, since $|X_0|$ is small compared with $P(\tilde T_\ve) = P(-T_\ve)$, the functions $(\tilde C(t), \tilde P(t))$ satisfy the identity
$$
\tilde C^{\la_0}(t)(\frac \la{\la_0} - \tilde C(t))^{1-\la_0}  =  c_\infty^{\la_0} (\frac \la{\la_0}- c_\infty)^{1-\la_0} \frac{a^p(\ve \tilde P(t))}{a^p(\ve \tilde P(\tilde T_\ve))}  =   c_\infty^{\la_0} (\frac \la{\la_0}- c_\infty)^{1-\la_0}  a^p(\ve \tilde P(t)) (1+O(\ve^{10})).
$$
Consider the smooth function $C>0 \mapsto f(C) := C^{\la_0} (\frac \la{\la_0} -C)^{1-\la_0} $. Using (\ref{cinfla}), we get
$$
| f(\tilde C(t)) -  f(C(t))| \leq K\ve \  a'(\ve P(t)) |\tilde P(t) -P(t)|  + K \ve^2 |P(t)-\tilde P(t)|^2.
$$
Note that $f(C)$ has nonzero derivative provided $C\neq \la$,\footnote{More specifically, $f'(C) = - (C-\la)C^{\la_0-1}(\frac \la{\la_0} -C)^{-\la_0}.$ }. Since $|C(t) -\la | \geq \kappa \al>0$, $\kappa>0$, uniformly in $\ve$ in the considered time region, we get
$$
|\Delta C(t)| \leq  K(\al) [ \ve  e^{ -\ga \ve |P(t)|}|\Delta P(t)|  +  \ve^2 |\Delta P(t)|^2],
$$
where $\Delta C(t) :=C(t) -\tilde C(t)$ and $\Delta P(t) :=P(t) -\tilde P(t) $. Now we recall that $\Delta C(t) = \Delta P'(t)$. Integrating $[t, \tilde T_\ve]$, with $t\geq t_0 +\frac \al\ve$, we get
$$
|\Delta P(t) | \leq |\Delta P(\tilde T_\ve)| + \int_{t}^{\tilde T_\ve}  K \ve  e^{ -\ga \ve |P(s)|}|\Delta P(s)|ds  + K\ve^2 \int_{t}^{\tilde T_\ve}|\Delta P(s)|^2ds.
$$
By the Gronwall inequality, one has $|\Delta P(t)| \leq K|\Delta P(\tilde T_\ve)|$ and $|\Delta C(t)| \leq K \ve |\Delta P(\tilde T_\ve)|$, as desired. Now we consider the interval $[t_0-\frac \al\ve, t_0+\frac \al\ve]$. From (\ref{c}) and (\ref{ct}),
$$
|\Delta P(t) | \leq |\Delta P(t_0 + \frac \al\ve) | + \int_t^{t_0 +\frac \al\ve} |\Delta C(s)|ds,
$$
and
$$
|\Delta C(t) | \leq |\Delta C(t_0 +\frac \al\ve)| +  K \ve \int_t^{t_0+\frac \al\ve} e^{-\ga\ve|P(s)|} (| \Delta C(s)| + \ve | \Delta P(s)|)ds.
$$
Hence one has $|\Delta C(t) | \leq K|\Delta C(t_0 +\frac \al\ve)| $ and $ |\Delta P(t)| \leq K |\Delta P(t_0+\frac \al\ve)| $. Finally, the  proof in the interval $t\in [-T_\ve,  t_0-\frac \al\ve]$ is similar to the first case. The proof is complete.
\end{proof}

We assume $(\tilde c(t), \tilde \rho(t))$ and $(\tilde C(t), \tilde P(t))$ satisfying (\ref{r1}). The following is the equivalent to Proposition \ref{prop:decomp} (see also \cite{Mu2}):

\begin{prop}\label{prop:decomp1} Let $(\tilde c(t), \tilde \rho(t))$ and $(\tilde C(t),\tilde P(t))$ be satisfying (\ref{r1}). There exists a constant $\ga>0$, independent of $\ve$ small, and an approximate solution $\tilde v$ of the form (\ref{hatv}), such that for all $t\in [-T_\ve, \tilde T_\ve]$, the following properties are satisfied.

\begin{enumerate}
\item The error term  $S[\tilde v]$ satisfies the decomposition
\bea\label{Decompv}
S[\tilde v](t,x) & = &   (\tilde c'(t) - \ve \tilde f_1(t) -\ve^2 \delta_{m,3} \tilde f_3(t))\partial_{\tilde c}\tilde v  \nonu\\
& & +\  (\tilde \rho'(t) - \tilde c(t)+ \la -  \ve \tilde f_2(t) -\ve^2 \delta_{m,3}\tilde f_4(t)) \partial_{\tilde \rho} \tilde v  + \tilde S[\tilde v](t,x). 
\eea

\item The functions $\tilde A_{\tilde c}, \tilde B_{\tilde c}$ are as follows: 
\be\label{tAc}
\tilde A_{\tilde c}, \partial_{\tilde c} \tilde A_{\tilde c} \in L^\infty(\R), \quad \tilde A_{\tilde c}'\in \mathcal Y,  \quad
|\tilde A_{\tilde c}(\tilde y) | \leq K e^{\ga \tilde y} \; \hbox{ as }\; \tilde y\to -\infty, \quad \lim_{+\infty} \tilde A_{\tilde c}  \neq 0,
\ee
and for $m=3$,
\be\label{tBc}
\begin{cases}
B_{\tilde c}'(t,\cdot) \in L^\infty(\R),\quad  |B_{\tilde c}(t, \tilde y) | \leq K e^{\ga \tilde y}e^{-\ve \ga|\rho(t)|} \; \hbox{ as }\; \tilde y\to -\infty, \\
|B_{\tilde c}(t,\tilde y)| +|\partial_{\tilde c} B_{\tilde c} (t,\tilde y)| \leq K|\tilde y| e^{-\ve \ga|\rho(t)|}, \; \hbox{ as } \;\tilde y\to +\infty. 
\end{cases}
\ee

\item\label{MA} The function $\tilde \eta_\ve \tilde w(t)$, with $\tilde w(t)$ defined in (\ref{tww}), satisfies similar estimates as in (\ref{H1})-(\ref{AO}).

\item In addition, $\tilde f_1 (t)=f_1(\tilde c(t), \tilde \rho(t))$, given by (\ref{f1}),
\be\label{tf2}
\tilde f_2(t) =-f_2(\tilde c(t), \tilde \rho(t)),  \quad \tilde f_3(t) = -f_3(\tilde c(t), \tilde \rho(t)),
\ee
and $\tilde f_4(t)$ satisfies a similar decomposition as (\ref{f4}).

\item Finally, $\tilde S[\tilde v] (t,\cdot )$ describes similar estimates as in (\ref{SH2}), (\ref{SIn}) and (\ref{SIn3}).
\end{enumerate}
\end{prop}

\begin{rem}
Let us emphasize the main differences between Propositions \ref{prop:decomp} and \ref{prop:decomp1}. Contrary to (\ref{Ac}) and (\ref{Bc}), we impose the {\bf opposite behavior} in (\ref{tAc})-(\ref{tBc}). This last condition is mainly motivated by the fact that the solution $v(t)$ is now pure as $t\to +\infty$, therefore it should be rapidly decaying on the {\bf left} hand side of the soliton, instead of the right one.   As a consequence, we get that the values of $\tilde f_2(t)$ and $\tilde f_3(t)$ are of opposite sign (cf. (\ref{tf2}).) 
\end{rem}

\begin{proof}[Sketch of proof of Proposition \ref{prop:decomp1}]
We follow step by step the proof of Proposition 4.2 in \cite{Mu3} (see also Appendix B in \cite{Mu3}), having in mind the following formal changes:
\bee
& & (C(t),P(t))  \mapsto  (\tilde C(t), \tilde P(t)),  \quad  (c(t),\rho(t))  \mapsto  (\tilde c(t),\tilde \rho(t)), \quad \tilde u(t) \mapsto  \tilde v(t),\\
& &   (f_1(t),f_2(t),f_3(t),f_4(t)) \mapsto   (\tilde f_1(t), \tilde f_2(t), \tilde f_3(t),\tilde f_4(t)).
\eee

\noindent
{\bf Steps 0, 1, 2, 3 and 4.} In these paragraphs, no significant modifications are needed. Let us recall that $F_1$ is given by
\be\label{F11}
F_1 =  \frac{\tilde f_1(t)}{\tilde a(\ve \tilde \rho)} \Lambda Q_{\tilde c} +\frac{a'}{\tilde a^m}(\ve \tilde \rho) \big[ (\tilde yQ_{\tilde c}^m)_{\tilde y}- \frac 1{m-1} (\tilde c-\la)Q_{\tilde c}  \big] -  \frac{\tilde f_2(t)}{\tilde a(\ve \tilde\rho)} Q_{\tilde c}' ,
\ee
in particular, $\tilde f_1(t) =f_1(\tilde c(t),\tilde \rho(t))$. The term $F_2$ remains ``unchanged''.

\smallskip

\noindent
{\bf Step 5. Resolution of the first linear problem.} We are looking for a function $\tilde A_{\tilde c}$ with the opposite behavior with respect to $A_c$ (cf. (\ref{tAc})). The key difference will be in the computation of $\tilde f_2(t)$. Indeed, we start from (B.29) in \cite{Mu3}. We have (for the sake of clarity, we drop the variable $t$ and the tilde on each function, if there is no confusion)  
$$
\int_\R (\mathcal L A_c)_y \int^{+\infty}_{y} \Lambda Q_c = \int_\R (\tilde F_1 +\la F_2) \int^{+\infty}_{y} \Lambda Q_c , \qquad \mathcal L := -\partial_{yy} +c -mQ_{c}^{m-1}.
$$
and therefore, using that we have chosen $\int_\R Q_c A_c =0$ \cite{Mu3} and $\Lambda Q_{c} :=\partial_c Q_c,$
$$
 \int_\R (\mathcal L A_c)_y \int^{+\infty}_{y} \Lambda Q_c = (\mathcal L A_c) \int^{+\infty}_{y} \Lambda Q_c \Big|_{-\infty}^{+\infty}   + \int_\R \Lambda Q_c \mathcal L A_c  = -\int_\R Q_c A_c =0.
$$
Note that, in this case, we have used that $(\mathcal L A_c) \int^{+\infty}_{y} \Lambda Q_c \Big|_{-\infty}^{+\infty} =0$. Therefore, from (\ref{F11}) and \cite[B.20]{Mu3}, 
$$
f_2 \int_\R Q_c \Lambda Q_c = \frac{a'}{a}\int_\R \Big[ p  c (c -  \frac \la{\la_0 } ) \Lambda Q_c   - \frac 1{m-1} (c-\la)Q_c +(yQ_c^m)' \Big]\int^{+\infty}_y \Lambda Q_c.
$$
Using that  $\int^{+\infty}_y \Lambda Q_c = -\int_0^y \Lambda Q_c + \int_0^{+\infty} \Lambda Q_c  =  -\int_0^y \Lambda Q_c + \frac 12\int_\R \Lambda Q_c, $
one has
$$
\theta  f_2 c^{2\theta -1} \int_\R Q^2   =  \frac{a'}{2a} \Big[  p  c (c -  \frac \la{\la_0 } ) \int_\R \Lambda Q_c   - \frac 1{m-1} (c-\la) \int_\R  Q_c \Big]\int_\R \Lambda Q_c ,
$$
and therefore
$$
f_2(t) = -\frac{(3-m)}{(5-m)^2}(3\la_0 c -\la)\frac{a'(\ve \rho)}{\sqrt{c}a(\ve \rho)} \frac{(\int_\R Q)^2}{\int_\R Q^2},
$$
as desired (cf. (\ref{tf2})).

\noindent
{\bf Step 6. Cubic case. resolution of a second linear system.} As above, the main difference here is in the value of $\tilde f_3(t)$, which is the ``opposite'' of $f_3(t)$. This result is consequence of  (\ref{tBc}).

We start from the equivalent of \cite[B.35]{Mu3} in our case. The first big difference is in (B.36). Now we have
$$
\int_\R A_c \mathcal L A_c'   =    A_c \mathcal L A_c \Big|_{-\infty}^{+\infty} -\int_\R A_c (  \tilde F_1 + \la \hat F_1) 
 =  cA_c^2(+\infty)- \int_\R A_c (  \tilde F_1 + \la \hat F_1),
$$
and therefore,
$$
3 \int_\R Q_cQ_c' A_c^2 = \frac 12 cA_{c}^2(+\infty) -\int_\R  A_c (  \tilde F_1 + \la \hat F_1).
$$
In the same way, $\mu_c  = \frac 12 cA_{c}^2(+\infty)  -\int_\R A_c Q_c^3.$
Still following the proof of (B.35), we have that
\bee
\mu_c &= & \frac 12 cA_{c}^2(+\infty)  +\frac 12 \int_\R \mathcal L A_c Q_c = \frac 12 cA_{c}^2(+\infty)  + \frac 12 \int_\R Q_c\int_{-\infty}^y (\tilde F_1 + \la \hat F_1) \\
& =& \frac 12 cA_{c}^2(+\infty) + \frac 14 \int_\R Q_c \int_\R (\tilde F_1 + \la \hat F_1).
\eee
Since $A_c(+\infty) = \frac 1c\int_\R (\tilde F_1 + \la \hat F_1)   = -\frac 1{2c}(c-\la)\int_\R Q_c,$
we finally get $\mu_c = -\frac\la{8c} (c-\la)(\int_\R Q)^2.$
Therefore, $\tilde f_3(t) = -f_3(\tilde c(t), \tilde \rho(t)).$

\smallskip

\noindent
{\bf Step 7. Final conclusion.} No differences, apart from the obvious ones, are present in this paragraph. The sketch of proof of Proposition \ref{prop:decomp1} is now complete.
\end{proof}

In the following lines, we state without proof the equivalent of Proposition \ref{prop:I} for the solution $v(t)$.

\begin{prop}\label{T0v} Suppose $0< \la <1$, $\la\neq \tilde \la$. There exists a constant $\ve_0>0$ such that the following holds for any $0<\ve <\ve_0$. There are a constant $K>0$ independent of $\ve$ and unique $C^1$ functions $\tilde c, \tilde \rho : [-T_\ve, \tilde T_\ve] \to \R$ such that, for all $t\in [-T_\ve, \tilde T_\ve]$, the function $\tilde z(t,x) := v(t) - \tilde v(t \  ; \tilde c(t), \tilde \rho(t))$
satisfies
\be\label{INT41v}
\|\tilde z(t) \|_{H^1(\R)} \leq K \ve^{1/2}, \quad \int_\R \tilde z(t,x) Q_{\tilde c}(\tilde y)dx = \int_\R \tilde yQ_{\tilde c}(\tilde y) \tilde z(t,x)dx =0.
\ee
\end{prop}

From the proof of this result one can obtain several additional properties, as in Proposition \ref{prop:I}. We recall some of them, of importance in the following lines. First of all, $\tilde z(t)$ satisfies the gKdV equation
\be\label{tEqz1}
  \tilde z_t +  \big\{ \tilde z_{xx}  -\la \tilde z +  a(\ve x) [ (\tilde v + \tilde z)^m - \tilde v^m ] \big\}_x  +  \tilde S[\tilde v]     + \tilde c_1'(t) \partial_{\tilde c} \tilde v  + \tilde \rho_1'(t) \partial_{\tilde \rho} \tilde v =0, 
\ee
with
$\tilde c_1' := \tilde c'-\ve \tilde f_1-\ve^2 \delta_{m,3} \tilde f_3$, and $\tilde \rho_1' := \tilde\rho' - \tilde c+\la- \ve \tilde f_2-\ve^2 \delta_{m,3} \tilde f_4.$ Second, there exists $\ga>0$ such that, for every $t\in [-T_\ve, \tilde T_\ve]$,
\be\label{trho1}
 |\tilde \rho_1'(t) |   \leq    K (m-3 +\ve e^{-\ga\ve|\tilde \rho(t)|} ) \Big[\int_\R \tilde z^2(t) e^{- \ga |\tilde y|} \Big]^{1/2}  +  K\int_\R e^{- \ga |\tilde y|}\tilde z^2(t)   +K \abs{\int_\R \tilde y Q_{\tilde c} \tilde S[\tilde v]},
\ee
\be\label{tc1}
|\tilde c_1'(t) |   \leq  K  \int_\R e^{-\ga |\tilde y|} \tilde z^2(t)  +  K \ve e^{-\ga\ve|\tilde \rho(t)| } \Big[ \int_\R e^{-\ga  |\tilde y|} \tilde z^2(t)\Big]^{1/2}  + K\abs{\int_\R Q_{\tilde c} \tilde S[\tilde v]},
\ee
and 
\be\label{allT}
|\tilde c(t)- \tilde C(t)| + |\tilde \rho'(t) -\tilde P'(t)| \leq K\ve^{1/2},
\ee
\be\label{CTT}
|\tilde c(\tilde T_\ve) - c^+| + | \tilde \rho(\tilde T_\ve) - \tilde P( \tilde T_\ve) | + \|z(\tilde T_\ve)\|_{H^1(\R)} \leq K\ve^{10},
\ee
with $K>0$ independent  of $\ve$. This information allows us to prove a Virial identity for $\tilde z$, as in Lemma  \ref{VL} (see \cite{Mu3} for the proof).

\begin{lem}\label{VLtil} There exist $K, A_0, \delta_0>0$  such that for all $t\in [-T_\ve, \tilde T_\ve]$ and for some $\ga =\ga(A_0)>0$,
\be\label{dereta2}
 \partial_t \int_\R  \tilde z^2(t,x) (1-\psi_{A_0})(\tilde y)   \geq   \delta_0  \int_\R ( \tilde z_x^2 + \tilde z^2 )(t,x) e^{-\frac 1{A_0} | \tilde y|}  - KA_0 \ve^{5/2}.
\ee
\end{lem}

As in (\ref{intc1}), this last property leads to the estimate
\be\label{dereta3}
\int_{t}^{\tilde T_\ve} \int_\R ( \tilde z_x^2 + \tilde z^2 )(s,x) e^{-\frac 1{A_0} | \tilde y|}dxds  + \int_{t}^{\tilde T_\ve} |\tilde c_1'(s)|ds \leq K\ve^{3/2-1/100} ,
\ee
for $t\in [-T_\ve, \tilde T_\ve]$, and where we have used that $1-\psi_{A_0}>0$ and (\ref{CTT}).

\medskip

\section{Backward stability}\label{5}

Let $\delta>0$ a small number, to be chosen below. In this section we will assume that, for $T\geq T_\ve$ large enough, one has 
\be\label{SmaaD}
\|u(T) - v(T) \|_{H^1(\R)} \leq  K\nu \ve^{1 +\delta},
\ee
with $\nu$ a small number, to be specified below, and $K$ a fixed constant. We claim that this smallness condition is preserved for all time below $T$, in particular along the time interval $[-T_\ve, \tilde T_\ve]$.

\begin{lem}\label{BaSta} Suppose $\la\in(0,1)$, $\la\neq \tilde \la$, and $\delta>0$ small. There exist $K>0$ and a smooth function $\mathcal T(t) \in \R$, defined for all $t\in [-T_\ve,T]$, such that
\be\label{Smaa0}
\|u(t+\mathcal T(t)) - v(t) \|_{H^1(\R)} + |\mathcal T'(t)| \leq  K \nu \ve^{1 +\delta}.
\ee
\end{lem}

\begin{rem}
Let us emphasize that the modulation via the function $\mathcal T(t)$ is in part  consequence of the fact that \emph{there is no space invariance} for the equation (\ref{aKdV0}), and therefore modulation in space is not enough, in particular inside the interaction region. This idea has been previously introduced in \cite{Mu2}. 
\end{rem}

\begin{proof}[Proof of Lemma \ref{BaSta}]
We sketch the proof of this result, since it is similar to the proof of Proposition 2.5 in \cite{Mu3}, and Proposition 5.1 in \cite{Mu2}. We proceed in two steps.

\smallskip

\noindent
{\bf First step: From $t=T$ to $t=\tilde T_\ve$.}  We claim that for all $t\in [\tilde T_\ve, T]$, there exists $\mathcal T(t)\in \R$ such that
$$
\|u(t + \mathcal T(t) ) - v(t) \|_{H^1(\R)}  + |\mathcal T'(t)|\leq  K \nu \ve^{1+\delta},
$$
with $K>0$ independent of $\ve, \nu$ and $t$. Indeed, we define, for $K^*>0$ to be fixed later,
\bea\label{Tetoile}
T^* & := & \inf\{ t\in [\tilde T_\ve, T] \ \hbox{such that, for all} \ t'\in [t,  T],\ \hbox{there exists a smooth} \nonu \\
&  & \qquad  \tilde T(t) \in \R \hbox{ satisfying} \ \|u(t+\tilde T(t)) -v(t) \|_{H^1(\R)} \leq K^* \nu \ve^{1+\delta}    \}.
\eea
We suppose that $T^* >\tilde T_\ve$. We define, via the implicit function theorem, functions $\mathcal T(t)$ and $h(t)$, such that $h(t) := u(t +\mathcal T(t)) - v(t) $ satisfies, for all $t\in [T^*,T]$,
\be\label{HH}
\int_\R h(t,x)v_x(t,x) dx=0.
\ee
In addition, one has $\|h(t)\|_{H^1(\R)} + |\mathcal T'(t)|   \leq KK^* \nu \ve^{1+\delta}$, for some positive constant $K$. Additionally, this estimate at time $t=T$ does not depend on $K^*$. Define a mass $\tilde M[u](t)$ as follows:
$$
\tilde M[u](t) :=
\begin{cases}
\hat M[u](t) & c_\infty(\la)<\la, \quad (\hbox{cf. } (\ref{hM})),  \\
\mathcal M[u](t)   & c_\infty(\la)>\la, \quad (\hbox{cf. } (\ref{Mback})).  
\end{cases}
$$
Note that this quantity satisfies, for all $t\in [T^*, T]$,
\be\label{MonoM}
\begin{cases} 
\tilde M[u](t) -\tilde M[u](T) \geq 0,  & c_\infty(\la)<\la, \quad \hbox{(cf. (\ref{hM})-(\ref{hM3}))}, \\
\tilde M[u](t) -\tilde M[u](T) \leq K e^{-\ga\ve t}, & c_\infty(\la)> \la, \quad \hbox{(cf. Lemma A.2 \cite{Mu3})}.
\end{cases}
\ee
Moreover, this result does not vary if we consider instead the translated mass $\tilde M[u(\cdot +\mathcal T(\cdot))](t)$.

On the one hand, since $\tilde M[v](t) = \kappa_{m}(\la) M[Q] + O(e^{-\ve\ga t})$, with $\kappa_{m}(\la)$ a positive constant (recall that $v(t)$ is a \emph{pure} soliton solution at $+\infty$), and $E_a[v](t) =E_a[v](T)$, one has
$$
| E_a[v](T) - E_a[v](t) + (c_\infty(\la) - \la)(\tilde M[v](T) - \tilde M[v](t) ) | \leq K e^{-\ve\ga t}. 
$$
On the other hand, from the decomposition $u(t+\mathcal T(t)) = v(t) + h(t) $, we get
\be\label{Bal0}
E_a[u(\cdot +\mathcal T(\cdot )) ](t)  + (c_\infty(\la) - \la) \tilde M[u(\cdot +\mathcal T(\cdot )) ](t)   = E_a[v](t)  + (c_\infty(\la) - \la) \tilde M[v](t)  + \tilde F(t),
\ee
with $\tilde F(t)$ a coercive Weinstein functional in $h(t)$ (see e.g. \cite[Lemma 2.2]{Mu3}), up to a negative direction represented by  $v(t)$. This direction can be controlled using the energy conservation law for $E_a[u](t)$, as is done in \cite[Lemma 5.4]{Mu2}. Indeed, note that 
\begin{align}\label{Bal1}
E_a[u](T) & =  E_a[u(\cdot +\mathcal T(\cdot))](t)  \ =   E_a[v](t) - \int_\R h (v_{xx} -\la v +a(\ve x) v^m) + O(\|h(t)\|_{H^1(\R)}^2)\nonu \\
& = E_a[v](T) - (c_\infty -\la)\int_\R v h(t)  - \int_\R h (v_{xx} -c_\infty v +a(\ve x) v^m) + O(\|h(t)\|_{H^1(\R)}^2). 
\end{align}
Since $c_\infty(\la)\neq \la$ for all $\la\neq \tilde \la$, we have
$$
\abs{\int_\R vh(t) -\int_\R vh(T) } \leq K(\la)  K^*\nu \ve^{1 +\delta}  ( e^{-\ga \ve t} + K^*\nu \ve^{1+\delta}).
$$
We evaluate (\ref{Bal0}) at $t=T$ and $t=T^*$, and use this last estimate. From the coercivity of $\tilde F(t) $ up to the direction $v(t)$ we get, for $\ve_0>0$ small enough,
$$
\|h(T^*)\|_{H^1(\R)}^2 \leq K(\la) K^* \nu^2\ve^{2(1+\delta)} +  K e^{-\ga \ve t} \leq \frac 12(K^*)^2 \nu^2\ve^{2(1+\delta)},
$$
for $K^*$ large, independent of $\ve$ and $\nu$, which is a contradiction to the definition of $T^*$. This proves the first step of the proof.

\smallskip

\noindent
{\bf Final step.} We prove the result inside the interval $[-T_\ve, \tilde T_\ve]$.  The proof is similar to the above case, but in this opportunity we start from the initial estimate 
$$
\|u(\tilde T_\ve + \hat T_\ve ) - v(\tilde T_\ve) \|_{H^1(\R)}  \leq  K \nu \ve^{1+\delta},\qquad \hat T_\ve := \mathcal T(\tilde T_\ve).
$$
Note that $u(t + \hat T_\ve )$ is also a solution of (\ref{aKdV0}), with same energy and the same pure asymptotics as $t\to -\infty$. Therefore, in what follows we can assume by simplicity that $\hat T_\ve =0$. We define (\ref{Tetoile}) in the same way, but now we work inside the interval $[-T_\ve, \tilde T_\ve]$. In a similar fashion, we define $h(t)$ and $\mathcal T(t)$, as in (\ref{HH}). However, the energy-mass argument above considered is not valid anymore, since the mass variation is too large; we need a different approach.

In order to savage our proof, we follow the argument of \cite[Proposition 5.1]{Mu2}. We consider the Weinstein functional
$$
\mathcal F(t) := \frac 12 \int_\R (h_x^2 + \tilde c(t) h^2) -\frac 1{m+1}\int_\R a(\ve x) [ (v+h)^{m+1} -v^{m+1} -(m+1)v^mh ],  
$$ 
with $\tilde c(t)$ being the scaling of the approximate solution $\tilde v(t)$, close to $v(t)$. This quantity $\mathcal F(t)$ varies slowly, as shows a direct computation, similar to \cite[Lemma 5.6]{Mu2}. In particular, the variation of $\tilde c(t)$ can be controlled using (\ref{dereta3}). The last step is a sharp control of the quantity 
\be\label{vh}
\int_\R vh,
\ee
better than the standard Cauchy-Schwarz inequality. This can be done using a similar argument as above (see also \cite[Lemma 5.4]{Mu2}), as long as $\tilde c(t) \neq \la$. This is certainly true in the case $0<\la<\tilde \la$, since $\tilde C(t)>\la$ (Lemma \ref{ODE}) and $|\tilde c(t) - \tilde C(t)| \leq K \ve^{1/2}$ for all $t\in [-T_\ve, \tilde T_\ve]$. Therefore, in the case $0<\la<\tilde \la$, we are done.

In order to control this quantity in the case $\tilde \la<\la<1$, we use the following argument. Suppose $t\geq T^* > t_0 +\frac \al\ve$, for $\al>0$ small and $t_0$ such that $\tilde C(t_0)=\la$. It is clear that one can control \eqref{vh} inside the interval $[T^*, \tilde T_\ve]$.  Indeed, following (\ref{Bal1}), one has
\be\label{Bal2}
\abs{\int_\R v h (t) -\int_\R vh(\tilde T_\ve)} \leq \frac K\al K^* \nu \ve^{1+\delta}  \big\{  \ve e^{-\ve\ga |\tilde\rho(t)|} + K^* \nu \ve^{1+\delta}\big\}, 
\ee
which improves the standard estimate, provided we take $\ve$ small, depending on $K^*$ and $\al$.  As a conclusion, $T^* \leq t_0 + \frac \al\ve$.

Now we suppose $T^* > t_0 - \frac \al\ve$, and we consider the control of \eqref{vh} inside the interval $[T^*,  t_0+ \frac \al\ve] $. Moreover, we may suppose $T^*<t_0$, which is the most difficult case, since $\tilde C(t_0)=\la$ and the usual estimate degenerates. However, since the interval is \emph{small}, one can use a standard balance of mass.  One has, for $t\in [T^*,  t_0+ \frac \al\ve],$
\be\label{BMass}
M[u(\cdot + \mathcal T(\cdot))](t) - M[v](t) = \frac 12 \int_\R [(v+h)^2(t) -v^2(t) ]= \int_\R vh(t) +\frac 12 \int_\R h^2(t).
\ee
On the other hand, since $u$ and $v$ are solutions of (\ref{aKdV0}), one has from (\ref{Ma}),
$$
 \partial_t \big\{ M[u(\cdot + \mathcal T(\cdot)](t) - M[v](t)  \big\} =  \frac{\ve}{m+1} \int_\R a' ((v+h)^{m+1} -v^{m+1})   + \frac{\ve \mathcal T'(t)}{m+1} \int_\R a' (v+h)^{m+1}.
$$
Therefore, after integration in $[t,t_0 +\frac \al\ve$], and using (\ref{BMass}), we get
\be\label{Bal3}
\abs{ \int_\R vh(t) -  \int_\R vh(t_0+\frac \al\ve) } \leq K K^* \nu \ve^{1+\delta}  ( \al  + K^* \nu \ve^{1+\delta}),
\ee
which improves the standard estimate, for $\al>0$ small enough, depending on $K^*$ (take e.g.  $\al = (K^*)^{-1/100})$. Therefore, $T^* \leq t_0 -\frac \al\ve$.

The final estimate inside the interval $[-T_\ve, t_0-\frac \al\ve]$ is completely analogous to (\ref{Bal2}). One has, for $t \in [T^*, t_0-\frac \al\ve]$,
\be\label{Bal4}
\abs{\int_\R v h (t) -\int_\R vh(t_0-\frac \al\ve)} \leq \frac K\al K^* \nu \ve^{1+\delta}  (  \ve e^{-\ve\ga |\tilde\rho(t)|} + K^* \nu \ve^{1+\delta}).
\ee
Combining estimates (\ref{Bal2})-(\ref{Bal4}), taking $K^*$ large and $\ve$ small, depending on $K^*$, we obtain a contradiction. The proof is complete.
\end{proof}

We recall that from Propositions \ref{prop:decomp1} and \ref{T0v}, there exists a suitable approximate solution $\tilde v(t)= \tilde v(t\ ; \tilde c(t), \tilde \rho(t)) $, defined for $t\in [-T_\ve, \tilde T_\ve]$, of the form (\ref{hatv}), with dynamical parameters $\tilde c(t)$  and $\tilde \rho(t)$.

The purpose in what follows is to use the smallness condition (\ref{Smaa}) to obtain upper bounds on the variation of parameters  $(c,\rho)$ and $(\tilde c,\tilde\rho)$. Define, for $t\in [-T_\ve, \tilde T_\ve]$, the following quantities:
\be\label{DcDp}
t_+ := t+ \mathcal T(t), \quad  \Delta c(t) := c(t_+ ) - \tilde c(t), \quad\hbox{ and } \quad \Delta \rho(t) :=  \rho(t _+ )  -  \tilde \rho(t).
\ee
We have supposed that $|\mathcal T(t)|\leq  \frac {T_\ve}{100} $ in the interval $[-T_\ve, \tilde T_\ve]$, in such a way that we still can use the decompositions of Propositions \ref{prop:I} and \ref{T0v}. Later we will improve this result.  The next result states that under the condition (\ref{Smaa}) the quantities $\Delta c(t)$ and $\Delta \rho(t)$ are also small, meaning that  almost equal solutions have close dynamical parameters.

\begin{lem}\label{6p2} Assume $|\mathcal T(t)|\leq \frac {T_\ve}{100}$ in the interval $[-T_\ve, \tilde T_\ve]$. There exists $K,\ve_0>0$ such that for all $0<\ve<\ve_0$ the following holds. Suppose that (\ref{SmaaD}) is satisfied inside the interval $[-T_\ve, \tilde T_\ve]$. Then, for all $t\in [-T_\ve, \tilde T_\ve]$,
\be\label{small}
|\Delta c(t) | + |\Delta \rho(t)| \leq K\nu \ve^{1+\delta} + K\ve^{10}.
\ee
\end{lem}

\begin{proof}
From Propositions \ref{prop:I} and \ref{T0v} we have the following decomposition 
\be\label{DiffA}
u(t _+,x) -v(t, x)   = \tilde u(t_+,x) -\tilde v(t,x) + z(t_+,x)  - \tilde z(t, x), 
\ee
where, for all $t\in [-T_\ve, \tilde T_\ve],$ $z(t,x)$ and $\tilde z(t,x)$ satisfy (\ref{INT41}) and (\ref{INT41v}), respectively. In addition, after a Taylor expansion, we obtain  
\bea\label{Diff0}
R(t_+ ,x) - \tilde R(t,x)  & = &  \frac{\Delta c(t)}{\tilde a(\ve \tilde \rho(t))} \Lambda Q_{\tilde c(t)}(\tilde y) - \frac{\Delta \rho(t)}{\tilde a(\ve \tilde \rho(t))} Q_{\tilde c(t)}'(\tilde y)  \nonu \\
& & + O_{H^1(\R)} ( \ve e^{-\ve\ga|\tilde \rho(t)|} |\Delta \rho(t)|+ |\Delta c(t)|^2 + |\Delta \rho(t)|^2) .   
\eea
On the other hand, from (\ref{AO}),
\be\label{Diff1}
 \abs{\int_\R \eta_\ve(y) w(t_+,x) Q_{\tilde c}(\tilde y) dx}  \leq  K\ve^{10} +K \|w(t_+)\|_{H^1(\R)} (|\Delta c(t)| + |\Delta \rho(t)| ),
\ee
and from (\ref{AO}) and item \ref{MA} in Proposition \ref{prop:decomp1}, similar estimates hold for $\abs{\int_\R \eta_\ve (y)w(t_+,x) \tilde y Q_{\tilde c}(\tilde y)dx}$, $\abs{\int_\R \tilde \eta_\ve (\tilde y) \tilde w(t,x) Q_{\tilde c}(\tilde y) dx }$, and  $\abs{\int_\R \tilde \eta_\ve (\tilde y) \tilde w(t ,x) \tilde yQ_{\tilde c}(\tilde y)dx}$.
Finally, from (\ref{INT41}) one has
\be\label{Diff4}
\abs{\int_\R  z (t,x)Q_{\tilde c}(\tilde y)dx}  \leq K\|z(t)\|_{H^1(\R)} (|\Delta c(t)| + |\Delta \rho(t)| ),
\ee
and the same result is valid for the integration against $\tilde y Q_{\tilde c} $. Now we conclude. Integrating (\ref{DiffA}) against $Q_{\tilde c}$ and $\tilde yQ_{\tilde c}$, and using (\ref{Smaa0}), (\ref{INT41}), (\ref{INT41v}), and (\ref{Diff0})-(\ref{Diff4}), we finally obtain (\ref{small}).
\end{proof}

\medskip

\section{Propagation of the defect}\label{6}

\smallskip

Now, we suppose $m=2$ or $m=4$. In the following lines, we introduce two quantities, $J(t)$ and $\tilde J(t)$, with small variation in time, and such that the defect clearly appears in the dynamics. Let us define
\bea
\chi_c(t,x) & := & \int_{-\infty}^y  \Lambda Q_c(s) ds, \qquad e(t) := (3\la_0c(t)-\la)\frac{\tilde a(\ve \rho(t))}{2\theta c^{2\theta-1}(t)M[Q]} ,\label{PaPa}\\
\tilde \chi_{\tilde c}(t,x) & := &  \int^{+\infty}_{\tilde y}  \Lambda  Q_{\tilde c}(s) ds, \qquad \tilde e(t):= (3\la_0\tilde c(t)-\la)\frac{\tilde a(\ve \tilde \rho(t))}{2\theta \tilde c^{2\theta-1}(t)M[Q]} \label{MaMa}.
\eea
It is clear that $\chi_c$ remains bounded as $y\to +\infty$, and it is exponentially decreasing as $y\to -\infty$. Similarly, $\tilde \chi_{\tilde c}$ has the opposite behavior as $\tilde y\to \pm\infty$. Finally, let us recall the notation introduced in (\ref{hatv}) and Proposition \ref{T0v}. Consider the functionals
$$
J(t) := e(t)\int_\R \chi_c (t,x) z(t,x) dx, \quad \hbox{ and } \quad \tilde J(t) :=\tilde e(t) \int_\R \tilde{\chi}_{\tilde c} (t,x) \tilde z(t,x) dx.
$$

\begin{lem} The functionals $J(t)$ and $\tilde J(t)$ are well defined for all $t\in [-T_\ve, \tilde T_\ve]$, and they satisfy
\be\label{bJ}
|J(t)|  + |\tilde J(t)| \leq K\ve^{1/4}.
\ee
\end{lem}
\begin{proof}
We only prove the estimate for $J(t)$, being the estimate for $\tilde J(t)$ similar (see Remark \ref{cuid} below).

Let $y_0>0$ be a large number, independent of $\ve$, to be chosen later.  Note that $\chi_c(y)$, with $y=x-\rho(t)$, is an exponentially decreasing function as $y\to -\infty$. From (\ref{INT41}) and the Cauchy-Schwarz inequality, one has 
$$
\abs{\int_{y\leq y_0} \chi_c(y) z(t,x)dx } \leq K y_0 \ve^{1/2}. 
$$
On the other hand, the region $\{y>y_0\}$ requires more care since $\chi_c$ does not converge to zero as $y\to +\infty$. Let us suppose by now that, for all $t\in [-T_\ve, \tilde T_\ve]$, $z(t,x)$ enjoys the following exponential decay property:
\be\label{zinf}
|z(t,x)| \leq K \ve^{1/4} e^{- \ga (x-\rho(t))}, \quad x\geq  \rho(t)+y_0,
\ee
for some $K,\ga>0$, independent of $\ve$. This implies that 
$$
\abs{\int_{y>y_0} \chi_c (y) z(t,x)dx}  \leq K e^{-\ga y_0} \ve^{1/4}.
$$
These two inequalities imply (\ref{bJ}), since $y_0>0$ does not depend on $\ve$ small.

Note that (\ref{zinf}) is consequence of (\ref{INT41}), the following Gagliardo-Nirenberg inequality
$$
|z(t, y + \rho(t))|  \leq  \|z(t, \cdot +\rho(t))\|_{L^2(\cdot \geq y)}^{1/2}\|z_x(t, \cdot +\rho(t))\|_{L^2(\R)}^{1/2}   \leq  K\ve^{1/4}\|z(t, \cdot+\rho(t))\|_{L^2(\cdot \geq y)}^{1/2},
$$
and provided we prove that, for some $K,\ga>0$, and for all $y\geq y_0$,
\be\label{Oz}
\|z(t,\cdot +\rho(t))\|_{L^2(\cdot \geq y)}^{1/2} \leq Ke^{-\ga y}.
\ee
The proof of this last estimate is a consequence of the following estimate (see  e.g. \cite[Lemma 7.3]{Mu2} for a similar result):

\begin{lem}\label{ExpDec} There exist $K,\ga,y_0>0$, independent of $\ve$, such that for all $t\in [-T_\ve, \tilde T_\ve]$, and for all $y\geq y_0$,
\be\label{Betaa}
\| u(t,\cdot + \rho(t))\|_{L^2(\cdot\geq y)}^2 \leq Ke^{-\ga y}.
\ee
\end{lem}

\begin{proof}
The proof of this result can be divided in two steps.

\smallskip

\noindent
{\bf Step one. Reduction to the case $(-\infty,-T_\ve)$.} From \eqref{rho1}, there exists $\sigma_1,\sigma_2,\sigma_3\in \R$ such that $-\la<\sigma_1<\sigma_2<\sigma_3< \inf_{t\in [-T_\ve, \tilde T_\ve]} \rho'(t)$, independent of $\ve$. Indeed, using (\ref{r1}) and the fact that $C(t)>0$ uniformly in $\ve$,
\be\label{rhopp}
\rho'(t) = c(t) -\la + O(\ve) > 0.9 \inf_{t\in [-T_\ve, \tilde T_\ve]} C(t) -\la > -\la.    
\ee
Then it is clear that we can find such numbers. Suppose $y_0>0$ large, but fixed, to be chosen later, $ s,t\in [-T_\ve, \tilde T_\ve]$, with $s\leq t$. Consider the modified mass
\be\label{tildeI}
\tilde I_{t,y_0}(s) := \frac 12\int_\R a^{1/m}(\ve x) u^2(s,x)\phi(\ell)dx,
\ee
with $\ell= \ell(s,t):= x- (\rho(t) + \sigma_1 (s-t)+ y_0)$ and $\phi(\ell) := \frac 2\pi \arctan (e^{\ell/K_0})$, with $K_0>0$ large to be chosen below. From the definition of $\sigma_3$, we have
\be\label{lessthan}
\rho(t) -\rho(s) -\sigma_3 (t-s) \geq 0. 
\ee
Let us consider now (\ref{tildeI}). We claim that for $y_0>0$ large but arbitrary,
\be\label{Alpha}
\tilde I_{t,y_0}(t)-\tilde I_{t,y_0}(-T_\ve) \leq K e^{-y_0/K}(1 - e^{- (T_\ve + t)/K}).
\ee
Indeed, a direct computation gives 
\bee
 \frac 12\partial_s \int_\R a^{1/m}(\ve x) \phi u^2 
& = & -\frac 32 \int_\R a^{1/m}\phi'  u_x^2   + \frac{m}{m+1} \int_\R a^{1/m +1}(\ve x)\phi' u^{m+1} \\
& &   +\frac 12 \int_\R u^2 a^{1/m}(\ve x) \big[ -(\sigma + \la)\phi '   + \phi^{(3)} \big]  \\
& & -\frac 32\ve \int_\R (a^{1/m})' (\ve x) \phi  u_x^2 - \frac \ve 2 \int_\R u^2 [ \la (a^{1/m})'   - \ve^2 (a^{1/m})^{(3)} ](\ve x) \phi  \\
& & + \frac 32\ve \int_\R u^2 \big[ \ve(a^{1/m})^{(2)}(\ve x) \phi'  +  (a^{1/m})' (\ve x)\phi''   \big].
\eee
In the last computation we have clearly defined  six terms. Let us study in detail  each one. In what follows we use the decomposition $u =\tilde u + z$, given by Proposition \ref{prop:I}.

First of all, one has
$$
 \int_\R \phi'  a^{1/m}u_x^2 =  \int_\R \phi'  a^{1/m}(\tilde u_x^2 + 2\tilde u_x z_x + z_x^2). 
$$
Recall from Proposition \ref{prop:decomp} that $\tilde u(s,x)$ 
is exponentially decreasing in the region $x \geq \rho(s)$, independent of $\ve$. Moreover, it is zero for $x\leq \rho(s) -\frac 3\ve$. On the other hand, $\phi'$ is exponentially decreasing away from zero. Therefore, one has e.g.
\bee
\abs{ \int_\R a^{1/m} \phi'  \tilde u_x^2} &  \leq & K\int_{\rho(s)-\frac 3\ve}^{\rho(t) + \sigma_2(s-t) + \frac 12 y_0 } e^{\ell/K} dx + K\int_{\rho(t) + \sigma_2(s-t) + \frac 12 y_0 }^{\infty} e^{-\frac 1K (x-\rho(s))}dx \\
 & \leq & K e^{-\frac 1K ((\sigma_2-\sigma_1) (t-s) + \frac 12 y_0)} + K e^{-\frac 1K (\rho(t)-\rho(s) -\sigma_2 (t-s) +\frac 12 y_0)} \\
 & \leq & K e^{-y_0/K} e^{- (t-s)/K} + K e^{-\frac 1K ( (\sigma_3-\sigma_2) (t-s) +\frac 12 y_0)}  \leq \  K e^{-y_0/K} e^{- (t-s)/K},
\eee
for some $K>0$, and where we have used (\ref{lessthan}). The same method con be applied to the term $\int_\R \phi'  \tilde u_x z_x .$ Hence, one has
$$
\int_\R  a^{1/m}\phi'  u_x^2 =  \int_\R a^{1/m}\phi'   z_x^2 + O (e^{-y_0/K} e^{- (t-s)/K}).
$$
Similarly, since $u=\tilde u +z$,
\bee
\abs{ \int_\R a^{1/m+1} \phi' u^{m+1}}  & \leq &   K \abs{  \int_\R \phi' \tilde u^{m+1}}  + K \ve^{(m-1)/2} \int_\R a^{1/m}  \phi'  z^2 \\
& \leq &K e^{-(t-s)/K} e^{-x_0/K} + K \ve^{(m-1)/2} \int_\R a^{1/m} \phi'  z^2.
\eee
On the other hand, since $\sigma+ \la>0,$ taking $K_0>0$ large if necessary,
$$
 \int_\R a^{1/m} u^2 \big[ -(\sigma + \la)\phi'   +\phi^{(3)} \big]  = - \frac 12(\sigma + \la)\int_\R a^{1/m} \phi' z^2 + O(e^{-y_0/K} e^{- (t-s)/K}),
$$
and
$$
  -\frac 32\ve \int_\R (a^{1/m})' (\ve x) \phi  u_x^2 - \frac \ve 2 \int_\R u^2 [ \la (a^{1/m})'   - \ve^2 (a^{1/m})^{(3)} ](\ve x) \phi \leq 0,
$$
provided $\ve$ is small.
Finally,
$$
\abs{\frac 32\ve \int_\R  \big[ \ve(a^{1/m})^{(2)}(\ve x) \phi'  +  (a^{1/m})' (\ve x)\phi''   \big] u^2} \leq K \ve e^{-(t-s)/K} e^{-y_0/K}  + K\ve \int_\R a^{1/m}  \phi' z^2.  
$$
After these estimates, it is easy to conclude that
$$
\frac 12\partial_t \int_\R a^{1/m}(\ve x)\phi(y)u^2 \leq K   e^{-y_0/K} e^{-(t-s)/K}.
$$
Therefore, estimate (\ref{Alpha}) follows after integration in time.

\noindent
{\bf Step two. Estimate in $(-\infty,-T_\ve)$.} Now we perform the same computation as above, but now inside the interval $(-\infty, -T_\ve)$. Indeed, it is not difficult to show that, for $t_0\leq -T_\ve \leq t,$
\be\label{menosTe}
\tilde I_{t,y_0}(-T_\ve) -\tilde I_{t,y_0}(t_0) \leq K e^{-y_0/K} (1 - e^{-( t-t_0)/K} ).
\ee
The final conclusion comes from the fact that  $\lim_{t_0\to -\infty}\tilde I_{t,y_0}(t_0) =0,$
as a consequence of (\ref{Minfty}). Collecting (\ref{Alpha}) and (\ref{menosTe}), we get
$$
\tilde I_{t,y_0}(t) \leq K e^{-y_0/K}.
$$
The proof of (\ref{Betaa}) is complete.
\end{proof}
Let us conclude the proof of (\ref{Oz}). From (\ref{INT41}) and (\ref{Betaa}) one has
$$
\int_{y +\rho(t)}^{+\infty} [ \tilde u^2 + 2  \tilde u z  + z^2](t, x)dx  \leq K e^{-\ga y}.
$$
Now we use the main properties of the decomposition of the function $\tilde u$, sated in Proposition \ref{prop:decomp}. One has, for $x\geq y +\rho(t)$, $y\geq y_0>0$ large,
\be\label{decayyy}
|\tilde u(t,x)| \leq K\abs{Q_c(x-\rho(t)) + \ve d(t) A_c(x-\rho(t)) } \leq K e^{-\ga (x-\rho(t))},
\ee
for some constants $K, \ga>0$, independent of $\ve$. Note that the fact that $A_c$ is exponentially decreasing for $x\geq \rho(t) +y_0$ is essential. Therefore, we finally get (\ref{Oz}):
$$
\int_{y +\rho(t)}^{+\infty} z^2(t, x)dx  \leq K e^{-\ga y}.
$$
\end{proof}
\begin{rem}\label{cuid}
Let us remark that the proof in the case of $\tilde J(t)$ is quite similar, with some basic changes. We need exponential decay of $\tilde z(t,x)$ on the {\bf left} side. Second, instead of $\phi$ one has to consider the function $1-\phi$, supported on the left side of the soliton, and since $\tilde A_{\tilde c} $ is exponentially decreasing for $x<\tilde \rho(t)$, estimate (\ref{decayyy}) holds for $\tilde v(t,x)$ in the region $x\leq y +\tilde \rho(t)$, $y\leq -y_0<0$ large.   
\end{rem}

Since $J(t)$ and $\tilde J(t)$ are well-defined, we can compute and estimate its variation in time.

\begin{lem} The functionals $J(t)$ and $\tilde J(t)$ satisfy, for some constants $K, \ga>0$, and for all $t\in [-T_\ve, \tilde T_\ve]$,
\be\label{dJ}
 \abs{ J'(t)  +   \rho_1'(t)(3\la_0c(t)-\la)  } \leq K \int_\R e^{-\ga\sqrt{c}|y|} z^2(t) + K\ve^{5/4} + K|c_1'(t)|.
\ee
and similarly for $\tilde J'(t)$:
\be\label{dJt}
 \big| \tilde J'(t)  -  \tilde \rho_1'(t)(3\la_0 \tilde c(t)-\la) 
 \big|  \leq K \int_\R e^{-\ga\sqrt{\tilde c}|\tilde y|} \tilde z^2(t) + K\ve^{5/4} + K|\tilde c_1'(t)|.
\ee
\end{lem}
\begin{proof}
Let us prove (\ref{dJ}).  We  compute:
\bee
J'(t) & = &  e' \int_\R \chi_c z+ e\int_\R \chi_c  z_t -\rho' e\int_\R \Lambda Q_c z + c' e \int_\R \partial_c \chi_c z \\
& =&  e' \int_\R \chi_c z  + e\int_\R \Lambda Q_c \big\{ z_{xx}  -c z +  a(\ve x) [ (\tilde u +z)^m - \tilde u^m ] \big\}  - \rho'_1e \int_\R \Lambda Q_c z + c_1' e\int_\R \partial_c \chi_c z \\
& &   + \ve f_1e \int_\R \partial_c \chi_c z    - \ve f_2e \int_\R \Lambda Q_c z - c'_1e \int_\R \chi_c  \partial_c \tilde u  -  e\int_\R \chi_c  \tilde S[\tilde u] - \rho_1' e \int_\R  \chi_c \partial_\rho \tilde u. 
\eee
Notice that we have used (\ref{Eqz1}). In the following lines, we estimate each term above. First of all, from (\ref{PaPa}), (\ref{INT41}) and (\ref{c1}),
\bee
|e'(t)| & \leq & K | c'(t)|  + K| a'(\ve\rho) \ve \rho'(t)|   \leq  K\ve e^{-\ga\ve|\rho(t)|} +K |c_1'(t)| \leq K\ve.
\eee
From here, we get $|e' \int_\R \chi_c z | \leq K\ve^{5/4}.$ Using (\ref{INT41}) and the identity $\mathcal L \Lambda Q_c:= -(\Lambda Q_c)''  + c\Lambda Q_c -  mQ_c^{m-1}\Lambda Q_c =-Q_c$, we have
\bee
& &  \abs{e\int_\R \Lambda Q_c \big\{ z_{xx}  -c z +  a(\ve x) [ (\tilde u +z)^m - \tilde u^m ] \big\} }   \leq  \\
& & \quad \leq    \abs{e\int_\R \Lambda Q_c \big\{\mathcal L z   + m [ a(\ve x) \tilde u^{m-1} - Q_c^{m-1}] z\big\} }    +  \abs{e\int_\R \Lambda Q_ca(\ve x) [ (\tilde u +z)^m - \tilde u^m - m\tilde u^{m-1}z]} \\
&  &   \quad \leq K\ve \|z(t)\|_{H^1(\R)} + K\int_\R e^{-\sqrt{c}|y|} z^2(t) \  \leq \ K\ve^{3/2} + K\int_\R e^{-\sqrt{c}|y|} z^2(t).
\eee
On the other hand, from (\ref{rho1}),
\bee
\abs{\rho_1' e \int_\R \Lambda Q_c z} & \leq &  K \int_\R e^{-\ga\sqrt{c}|y|} z^2(t) + K (\ve e^{-\ve\ga|\rho(t)|} +\|z(t)\|_{H^1(\R)}^2 +  \|\tilde S[\tilde u](t)\|_{L^2(\R)} )\|z(t)\|_{H^1(\R)}  \\
& \leq & K \int_\R e^{-\ga\sqrt{c}|y|} z^2(t)   + K\ve^{3/2}.
\eee
Note that $\partial_c \chi_c = \int_{-\infty}^{y} \partial_c \Lambda Q_c $ has a similar asymptotic behavior as $\chi_c$. Therefore, from the first part,
$$
\abs{ \ve f_1e \int_\R \partial_c \chi_c z } + \abs{ c_1' e \int_\R \partial_c \chi_c z} \leq K\ve^{5/4}.
$$
Similarly, $\abs{\ve f_2e \int_\R \Lambda Q_c z} \leq K\ve^{3/2}.$
Let us recall that, from Proposition \ref{prop:decomp}, $\tilde u$ is exponentially decreasing in $y$ as $y\to +\infty$, moreover $\tilde u\equiv 0$ for $y \leq -\frac 3\ve $. Since $\chi_c (y)$ is exponentially decreasing as $y\to -\infty$, one has, for some constant $\ga>0$,
\bee
\abs{c_1' \int_{\R} \chi_c  \partial_c \tilde u} & =& |c_1'|\abs{ \int_{y\geq -\frac 3\ve} \chi_c  (\frac 1{\tilde a} \Lambda Q_c + \ve d(t) \partial_c A_c) } \\
& \leq & K|c_1'| \abs{\int_\R \chi_c \chi_c'} + K \ve |c_1'|  \int_{-\frac 3\ve}^0 e^{\ga y} |\partial_c A_c(y)|dy +  K \ve |c_1'| \int_0^{+\infty} e^{-\ga y} dy  + K\ve^{10} \\
& \leq & K|c_1'|  + K\ve^{10}.
\eee
The term $\int_\R \chi_c  \tilde S[\tilde u] $ can be treated similarly. Indeed, since $\chi_c$ is exponentially decreasing as $y\to -\infty$, one has from Step 7 in Appendix B of \cite{Mu3},
\bee
\abs{\int_\R \chi_c  \tilde S[\tilde u] } & \leq & K \int_{-\frac 3\ve}^0 e^{\ga y}  |\tilde S[\tilde u]| +K \int_{y\geq 0} |\tilde S[\tilde u]|\\
& \leq & K  e^{-\ga/\ve}e^{-\ve \ga |\rho(t)|} + K \ve^2 e^{-\ve \ga |\rho(t)|} + K\ve^3 \ \leq \ K \ve^2 e^{-\ve \ga |\rho(t)|} + K\ve^3 \ \leq  \ K\ve^2.
\eee

\noindent
The last estimate concerns the nonzero term $\rho_1' e \int_\R  \chi_c \partial_\rho \tilde u $. Here one has
\bee
\rho_1' e \int_\R  \chi_c \partial_\rho \tilde u & = &   \frac{\rho_1'e}{\tilde a(\ve \rho)} \int_\R \Lambda Q_c Q_c + \ve \rho_1'  e \partial_\rho d(t) \int_\R \chi_c  A_c(y)  - \ve \rho_1'  e d(t) \int_\R \chi_c  A_c'(y)   \\
& =& \rho_1' (3\la_0c-\la) c^{2\theta}(2\theta M[Q])^{-1} \int_\R \Lambda Q_c Q_c   + O(|\rho_1'| \ve e^{-\ve\ga|\rho(t)|}) \\
& =& \rho_1'  (3\la_0c-\la)    + O(\ve^{3/2}).
\eee
Finally, collecting the above estimates, we get, for some $\ga>0$ independent of $\ve>0$ small,
$$
 \abs{J'(t)+ \rho_1' (3\la_0c-\la) } \leq  K \ve^{5/4} + K\int_\R e^{-\ga\sqrt{c}|y|} z^2(t) + K|c_1'(t)|,
$$
as desired. The proof of (\ref{dJt}) is analogous, the minus sign is a consequence of (\ref{MaMa}).
\end{proof}

\medskip

\section{Proof of Theorem \ref{MTL}, cases $m=2,4$}\label{7}

In this section we prove Theorem \ref{MTL} in the non degenerate cases $m=2$ and $4$. For the sake of clarity, we divide the proof into several steps.

\smallskip

\noindent
{\bf Step 1. Preliminaries.}  We will follow an argument by contradiction. Suppose that (\ref{LBound}) do not hold; therefore for $\nu >0$ arbitrarily small, there is $T>T_\ve$ arbitrarily large such that
\be\label{Tio}
\|w^+(T) \|_{H^1(\R)} \leq \nu \ve^{1+1/50}.
\ee
(cf. Theorem \ref{MTL1} for the definition of $w^+$ and $\rho(t)$). 
Let us define $ x_0 =x_0(T) := \rho(T) - (c^+ -\la) T$. From Proposition \ref{Existv} we know that there exists a unique solution $v= v_{x_0}$ of (\ref{aKdV}) such that (\ref{mlim0}) is satisfied. Moreover, from (\ref{mmTep}), by taking $T$ larger if necessary, one has  
$$
\|u(T) - v(T) \|_{H^1(\R)} \leq  2\nu \ve^{1+1/50}.
$$
Thanks to Lemma \ref{BaSta} with $\delta :=\frac 1{50}$, there are a constant $K>0$ and a smooth function $\mathcal T(t) \in \R$, defined for all $t\in [-T_\ve, \tilde T_\ve]$, such that
\be\label{Smaa}
\|u(t+\mathcal T(t)) - v(t) \|_{H^1(\R)} + |\mathcal T'(t)| \leq  K \nu \ve^{1+1/50}.
\ee

Now we assume that $\mathcal T(t)$ is a small perturbation of $T_\ve$ inside the interval $[-T_\ve, \tilde T_\ve]$, in the sense that 
$$
|\mathcal T(t)|\leq  K^* \ve^{-1/2-1/100}, \quad |\mathcal T(\tilde T_\ve)|\leq  K \ve^{-1/2-1/100},
$$
where $K^*>0$ is a large constant, to be chosen later, and $0<K<K^*$ is independent of $K^*$. Therefore Lemma \ref{6p2} makes sense with no modifications. Moreover, from Propositions \ref{prop:I} and \ref{T0v}, and (\ref{todo}), one has for $X_0 :=( c^+ -\la)\tilde T_\ve +x_0(T) -P(\tilde T_\ve) $,
$$
|X_0| \leq K|P'(\tilde T_\ve) | |\mathcal T(\tilde T_\ve)| +| P(\tilde T_\ve+ \mathcal T(\tilde T_\ve))  -( c^+ -\la)\tilde T_\ve - x_0(T) | \leq K \ve^{-1/2-1/100}.
$$
Note that we can apply  Lemma \ref{Finn}. As a consequence, we improve our previous assumption:

\begin{lem}[Bootstrap]
For all $\ve>0$ small, the function $\mathcal T(t)$ satisfies $|\mathcal T(t)| \leq \frac 12 K^*\ve^{-1/2-1/100} $ in the interval $[-T_\ve, T_\ve]$.
\end{lem}
\begin{proof}
We prove this result in the most difficult case, namely $\tilde \la<\la<1$. The case $0<\la<\tilde \la$ follows easily. Suppose that $t\in [-T_\ve, \tilde T_\ve]$, but $|t-t_0| \geq \frac \al\ve$, with $t_0$ given by Lemma \ref{ODE} and $\al>0$ small, independent of $\ve$. Then, from (\ref{small}), (\ref{todo}), (\ref{allT}) and (\ref{PtP}),
\bee
|\rho(t + \mathcal T(t)) - \rho(t)|  & \leq & |\Delta\rho(t)| +  | \tilde \rho(t) -\tilde P(t)| + |\tilde P(t)- P(t)| + |P(t) -\rho(t)| \\
&  \leq & K\nu\ve^{1+1/50} + K\ve^{-1/2-1/100} + |P(t) -\tilde P(t)| \ \leq \ K\ve^{1/2-1/100}.
\eee
A direct computation shows that, outside the interval $[t_0-\frac \al\ve, t_0+\frac \al\ve]$, one has $|C(t) -\la| > \kappa\al$, for some $\kappa>0$ independent of $\ve$. Therefore, from (\ref{rho1}) and (\ref{todo}),
$$
|\rho'(t)| \geq  |c(t) -\la| -K\ve^{1/2}  \geq  |C(t) -\la| - K\ve^{1/2} \geq \kappa \al/2 .
$$
Finally, from the lower bound $|\rho(t + \mathcal T(t)) - \rho(t)| \geq \frac 12 \kappa\al |\mathcal T(t)|,$ we get $|\mathcal T(t)|\leq K(\al)\ve^{1/2-1/100}$. Now we consider the estimate of $\mathcal T(t)$ inside the interval $[t_0-\frac \al\ve, t_0+\frac \al\ve]$. Since $|\mathcal T(t_0-\frac \al\ve)| \leq K \ve^{-1/2- 1/100}$, integrating (\ref{Smaa}), we get
$$
|\mathcal T(t)| \leq K \nu \ve^{1/100} + K \ve^{-1/2- 1/100} \leq  K \ve^{-1/2- 1/100}.
$$  
By taking $K^*>2K$, we can conclude. We are done.
\end{proof}

We have proved that $|\mathcal T(t)|$ is small, compared with $T_\ve$, in the interaction region. This means that, by performing a suitable translation in  time, we can assume, without loss of generality, that $\mathcal T(-T_\ve) =0$, and the arguments below do not change. 

\smallskip

\noindent
{\bf Step  2. Integration in time.} Using (\ref{Smaa}) and (\ref{tte}), we get, for $t\in [-T_\ve, \tilde T_\ve]$,
$$
|\mathcal T(t) | \leq \int_{-T_\ve}^t |\mathcal T'(s)|ds \leq K\nu \ve^{1/100},
$$
and thus from (\ref{DcDp}), (\ref{small}),   (\ref{rho1}) and (\ref{c1}),
\be\label{LaLb}
\begin{cases}
|c(t) - \tilde c(t)|  \leq |c(t)-c(t+\mathcal T(t))|+|c(t+\mathcal T(t)) -\tilde c(t)|   \leq  K \nu \ve^{1+1/100}, \\
 |\rho(t) - \tilde \rho(t)| \leq  | \rho(t)- \rho(t+\mathcal T(t))|+| \rho(t+\mathcal T(t)) -\tilde \rho(t)|\leq  K\nu \ve^{1/100}.
\end{cases}
\ee
Now we consider (\ref{dJ}) and (\ref{dJt}). Adding both inequalities, we have
\bea
& &  \abs{ J'(t)  +  \tilde J'(t) +   \rho_1'(t)(3\la_0c(t)-\la) - \tilde \rho_1'(t) (3\la_0\tilde c(t)-\la)    }  \nonu \\ 
& &  \quad \leq K \int_\R (e^{-\ga\sqrt{c}|y|} z^2(t) + e^{-\ga\sqrt{\tilde c}|\tilde y|} \tilde z^2(t)) + K\ve^{5/4}  + K(|c_1'(t)|   + |\tilde c_1'(t)|). \label{Lucho} 
\eea
Now we integrate between $-T_\ve$ and $\tilde T_\ve$. Using the Virial estimates (\ref{dereta}) and (\ref{dereta2}) for $A_0$ large enough, one obtains
$$
\int_{-T_\ve}^{\tilde T_\ve}  \int_\R (e^{-\ga\sqrt{c}|y|} z^2(t) + e^{-\ga\sqrt{\tilde c}|\tilde y|} \tilde z^2(t))dx dt \leq K\ve^{3/2-1/100},
$$
and similarly $\int_{-T_\ve}^{\tilde T_\ve} (|c_1'(t)| +|\tilde c_1'(t)|) dt \leq K\ve^{3/2-1/100}$ (see e.g. (\ref{Intec1})). 
On the other hand, from (\ref{bJ}),
$$
\Big|\int_{-T_\ve}^{\tilde T_\ve}  (J'(t)  + \tilde J'(t))dt\Big| \leq |J(t)| + |J(-T_\ve)| +|\tilde J(t)|+|\tilde J(-T_\ve)|\leq  K\ve^{1/4}.
$$
Hence, from (\ref{Lucho}), (\ref{LaLb}) and (\ref{rho1}),
$$
\Big| \int_{-T_\ve}^{\tilde T_\ve}(\rho_1'(t) - \tilde \rho_1'(t)) (3\la_0c(t)-\la) dt \Big| \leq  K \ve^{1/4- 1/100} + K\nu \ve^{1/2}.
$$
In addition, using (\ref{LaLb}),
$$
\Big| \int_{-T_\ve}^{\tilde T_\ve} (\rho_1'(t) - \tilde \rho_1'(t))(3\la_0c(t)-\la)  dt \Big|  \geq     c_m \Big|\int_{-T_\ve}^{\tilde T_\ve} \ve  (f_2(t) - \tilde f_2(t)) (3\la_0c(t)-\la)  dt\Big|  -K\nu. 
$$
We use (\ref{f2}), (\ref{tf2}), and (\ref{LaLb})  to obtain 
$$
\ve  \abs{\int_{-T_\ve}^{\tilde T_\ve}  \frac{( 3\la_0 c(t) - \la)^2}{c^{1/2}(t)}\frac{a'(\ve \rho(t))}{a(\ve \rho(t))}  dt } \leq K \ve^{1/4-1/100} + K\nu.
$$
Now we claim that the quantity in the left side is bounded below independent of $\nu $ and $\ve$, which gives the contradiction, for $\ve$ and $\nu$ small enough. Indeed, from (\ref{todo}), one has
$$
 \ve\int_{-T_\ve}^{\tilde T_\ve}  \frac{( 3\la_0 c(t) - \la)^2}{c^{1/2}(t)}\frac{a'(\ve \rho(t))}{ a(\ve \rho(t))}  dt = \ve\int_{-T_\ve}^{\tilde T_\ve} \frac{( 3\la_0 C(t) - \la)^2}{ C^{1/2}(t)}\frac{a'(\ve P(t))}{ a(\ve P(t))}  dt+ o_\ve(\ve).
$$
First, we consider the case $\la=\la_0$. In this case, from Lemma \ref{ODE} we have $C(t) \equiv 1$ and $P(t) =(1-\la_0)t$. Then
\bee
\ve\int_{-T_\ve}^{\tilde T_\ve} \frac{( 3\la_0 C(t) - \la)^2}{ C^{1/2}(t)}\frac{a'(\ve P(t))}{ a(\ve P(t))}  dt & = &  4\la_0^2 \int_{-T_\ve}^{\tilde T_\ve}  \frac{a'(\ve(1-\la_0) t)}{ a(\ve (1-\la_0)t)} \ve dt \nonu \\
& =& \frac{4\la_0^2}{1-\la_0} \log a(\ve (1-\la_0)t )\Big|_{-T_\ve}^{\tilde T_\ve} = \frac{4\la_0^2}{1-\la_0} \log 2 + o_\ve(1). \eee
It is clear that the last quantity is positive. Now we consider the general case, $\la\neq \la_0$. We have, from (\ref{c}),
\bea
  \ve\int_{-T_\ve}^{\tilde T_\ve} \frac{( 3\la_0 C(t) - \la)^2}{ C^{1/2}(t)}\frac{a'(\ve P(t))}{ a(\ve P(t))}  dt & =&  \frac{(5-m)}{4}\int_{-T_\ve}^{\tilde T_\ve}  \frac{( 3\la_0 C(t) - \la)^2}{C^{3/2}(t)( \la_0 C(t)-\la) } C'(t)  dt \nonu \\
& =& \frac{(5-m)}{4}\int_{1}^{c_\infty}   \frac{( 3\la_0 c - \la)^2}{c^{3/2}( \la_0 c-\la) } dc.\label{etat}
\eea
Note that the term inside the integral has always the same sign, and it is not identically zero for $0<\la<1$, $\la\neq \la_0$. Since (\ref{etat}) is always non zero, independent of $\nu$ and $\ve$, we get the final conclusion. The proof is complete.

\section{The cubic case}\label{8}

Consider now the proof of Theorem \ref{MTL} in the case $m=3$. This case is in some sense degenerate since $f_2\equiv 0$ in (\ref{f2}). Moreover, in this case $\chi_c (y) = \frac 12 yQ_c \in \mathcal S(\R)$ and from (\ref{INT41})-(\ref{INT41v}), one has that  the functionals $J(t)$ and $\tilde J(t)$ are identically zero. This is the reason why we needed to improve the approximate solution $\tilde u(t)$ (cf. Proposition \ref{prop:decomp} and \cite{Mu3}) to obtain a {\bf nonzero} defect in the solution. In this opportunity, a defect is given by the term $f_3(t)\neq 0$ in (\ref{f3}).

So, in order to prove the main result, instead of using the functionals $J(t)$ and $\tilde J(t)$, we consider the scaling laws (\ref{c1}) and (\ref{tc1}). Indeed, we start out following the Step 1 as in the previous section. Then we arrive to the estimate
\be\label{Smaa3}
\|u(t+\mathcal T(t)) - v(t) \|_{H^1(\R)} + |\mathcal T'(t)| \leq  K \nu \ve^{1+1/100},
\ee
valid for all $t\in [-T_\ve, \tilde T_\ve]$. Similarly, one has 
\be\label{aste}
|c(t) - \tilde c(t)|  \leq K \nu \ve^{1+1/100},\qquad |\rho(t) - \tilde \rho(t)| \leq K\nu\ve^{1/100}.
\ee

\begin{lem}\label{MaPa}

For all $t\in [-T_\ve, \tilde T_\ve]$, one has
\bea
 | c_1' (t)-\tilde c_1' (t) | &\leq &   K\ve^{5/2} + K\ve e^{-\ga\ve|\rho(t)|} \Big[\int_\R e^{-\ga |y|} z^2\Big]^{1/2} + K\ve e^{-\ga\ve|\tilde \rho(t)|}\Big[\int_\R e^{-\ga |\tilde y|} \tilde z^2\Big]^{1/2} \nonu\\
& &  \qquad + K\int_\R e^{-\ga |y|} z^2 + K\int_\R e^{-\ga |\tilde y|} \tilde z^2. \label{c1tc1}
\eea
\end{lem}

Let us assume the validity of this result and let us conclude the proof of Theorem \ref{MTL} for the cubic case. From \eqref{c1tc1}, \eqref{intc1} and \eqref{dereta3} one has, after integration and using the Cauchy-Schwarz inequality,
$$
\Big|\int_{-T_\ve}^{\tilde T_\ve}(c_1'(t)-\tilde c_1'(t))dt\Big| \leq  o_\ve(\ve).
$$
Therefore, from (\ref{Eqz1}), (\ref{tEqz1}) and (\ref{aste}),
\be\label{boundf3}
\Big|\ve^2 \int_{-T_\ve}^{\tilde T_\ve}f_3(t)dt\Big| \leq K\nu \ve + o_\ve(\ve).
\ee
Note that we have used (\ref{f1}) and (\ref{aste}) to obtain  
$$
\Big|\ve \int_{-T_\ve}^{\tilde T_\ve} (f_1(t)-\tilde f_1(t))dt\Big| \leq K\nu\ve.
$$
Coming back to (\ref{boundf3}), and using (\ref{f3}), one has
\be\label{Salfate}
\abs{\ve^2 \int_{-T_\ve}^{\tilde T_\ve} \frac 1{\sqrt{c(t)}}(c(t)-\la) \frac{a'^2(\ve\rho(t))}{a^2(\ve\rho(t))}dt} \leq K\nu \ve + o_\ve(\ve).
\ee
In what follows, we split the proof in two cases.

\noindent
{\bf First case: $0<\la<\tilde \la$.} From (\ref{rho1})-(\ref{c1}),
\be\label{integrale}
\ve\int_{-T_\ve}^{\tilde T_\ve} \frac{(c(t)-\la)}{\sqrt{c(t)}}  \frac{a'^2}{a^{2}}(\ve \rho(t)) dt \geq \frac{9}{10 \sqrt{c_\infty(\la)}} \int_{-1}^{1} \frac{a'^2}{a^{2}}(s)ds =:\kappa, 
\ee
with $\kappa>0$ independent of $\ve$. We get then $\tilde{\kappa} \ve \leq K\nu \ve  + K\ve^{3/2},$ for some positive constants $K,\tilde \kappa$. By taking $\nu$ small enough, we obtain the desired contradiction. This proves the result in the case $0<\la<\tilde \la$.

\noindent
{\bf Second case: $\tilde \la<\la<1$.} First of all, note that from (\ref{todo}), one has
$$
\ve\int_{-T_\ve}^{\tilde T_\ve} \frac{(c(t)-\la)}{\sqrt{c(t)}}  \frac{a'^2}{a^{2}}(\ve \rho(t)) dt= \ve\int_{-T_\ve}^{\tilde T_\ve} \frac{(C(t)-\la)}{\sqrt{C(t)}}  \frac{a'^2}{a^{2}}(\ve P(t)) dt + o_\ve(\ve).
$$
Now we split the time interval $[-T_\ve, \tilde T_\ve]$ into three disjoints subintervals, as in the proof of \cite[Lemma 3.3]{Mu3}. Let $t_0$ be as in Lemma \ref{ODE}. We have, for $\al>0$ small, independent of $\ve$,
\be\label{Pera3}
  \int_{-T_\ve}^{\tilde T_\ve} \frac{\ve (C(t)-\la)}{\sqrt{C(t)}}  \frac{a'^2}{a^{2}}(\ve P(t)) dt = \Bigg[ \int_{-T_\ve}^{t_0 -\frac \al\ve} + \int_{t_0 -\frac \al\ve}^{t_0 +\frac \al\ve} + \int_{t_0 +\frac \al\ve}^{\tilde T_\ve} \Bigg] \frac{\ve (C(t)-\la)}{\sqrt{C(t)}}  \frac{a'^2}{a^{2}}(\ve P(t)) dt
\ee
A simple computation shows that, inside the interval $[t_0-\frac \al\ve, t_0+\frac \al\ve]$,
$$
C(t) -\la = C(t_0)-\la + C'(t_0) (t-t_0) + O(\ve^2(t-t_0)^2) =O(\al),
$$ 
and thus,
$$
\Big| \int_{t_0 -\frac \al\ve}^{t_0 +\frac \al\ve} \frac{\ve (C(t)-\la)}{\sqrt{C(t)}}  \frac{a'^2}{a^{2}}(\ve P(t)) dt \Big| \leq \frac{K\al^2}{\sqrt{\la}}.
$$
Note that $\la>\tilde \la>0$. On the other hand,
$$
\int_{-T_\ve}^{t_0 -\frac \al\ve}\frac{\ve (C(t)-\la)}{\sqrt{C(t)}}  \frac{a'^2}{a^{2}}(\ve P(t))  dt \leq \frac{1}{\sqrt{C(t_0-\frac \al\ve)}} \int_{\ve P(-T_\ve)}^{\ve P(t_0-\frac\al\ve)}  \frac{a'^2}{a^{2}}(s) ds;
$$
and
$$
\int_{t_0 + \frac \al\ve}^{\tilde T_\ve}\frac{\ve (C(t)-\la)}{\sqrt{C(t)}}  \frac{a'^2}{a^{2}}(\ve P(t)) dt \leq -\frac{1}{\sqrt{C(t_0+\frac \al\ve)}} \int_{\ve P(\tilde T_\ve)}^{\ve P(t_0+\frac\al\ve)}  \frac{a'^2}{a^{2}}(s) ds.
$$
Recall that, by definition of $\tilde T_\ve$ (\ref{timeescape}), one has  $P(\tilde T_\ve) =P(-T_\ve)$. Moreover,
$$
\abs{\int_{\ve P(t_0+ \frac\al\ve)}^{\ve P(t_0-\frac\al\ve)}  \frac{a'^2}{a^{2}}(s) ds} \leq K \ve \big[ P(t_0- \frac\al\ve) - P(t_0+ \frac\al\ve)\big] \leq K \al^3,
$$
since $P'(t_0) =C(t_0)-\la=0$ and $P^{(3)}(t) =C''(t) =O(\ve^2)$. Therefore, we have for some $K>0$,
\bee
(\ref{Pera3}) & \leq & -\Big[\frac{1}{\sqrt{C(t_0+\frac \al\ve)}} -\frac{1}{\sqrt{C(t_0-\frac \al\ve)}}\Big] \int_{\ve P(-T_\ve)}^{\ve P(t_0-\frac\al\ve)}  \frac{a'^2}{a^{2}}(s) ds+ K\al^2 + K\al^3 \\
& \leq &  -K\al \int_{\ve P(-T_\ve)}^{\ve P(t_0)}  \frac{a'^2}{a^{2}}(s) ds + K\al^2 + K\al^3.
\eee
Since $|\ve P(t_0)| \leq K(\la)$ \cite[estimate (3.7)]{Mu3}, and $\ve P(-T_\ve) = -\ve^{-1/100} \ll -K(\la)$, one has
$$
(\ref{Pera3}) \leq  -K\al \int_{-K(\la)-1}^{-K(\la)}  \frac{a'^2}{a^{2}}(s) ds +K\al^2 + K\al^3  \leq  -\tilde K(\la)\al <0,
$$
for some $\tilde K(\la)>0$, and $\al>0$ small, but independent of  $\ve$ and $\nu$. Replacing in (\ref{Salfate}), we get the desired contradiction, provided $\nu$ and $\ve$ are small enough. The proof is complete.

\begin{proof}[Proof of Lemma \ref{MaPa}]
We start from the proof of \cite[identity (4.47)]{Mu3}, where, for each $t\in [-T_\ve, \tilde T_\ve]$, one has
\bee
 c_1' \int_\R Q_c \partial_c \tilde u & = & -   \int_\R Q_c \tilde S[\tilde u] - \rho_1' \int_\R Q_c \partial_\rho \tilde u    -\rho_1' \int_\R  Q_c' z + c_1' \int_\R  \Lambda Q_c z   -\ve^2 f_4 \int_\R  Q_c' z  \\
 & & + \ve( f_1 +\ve f_3)\int_\R  \Lambda Q_c z    + \int_\R Q_c'  a(\ve x) [ 3\tilde u z^2 + z^3 ] + 3 \int_\R Q_c'  [a(\ve x)\tilde u^{2} - Q_c^{2}]z. 
 \eee
This expression leads to the bound (\ref{c1}) above. We also recall that a completely similar expression holds for $\tilde c_1(t)$.  On the one hand, from (\ref{rho1}) with $m=3$, and (\ref{SIn3}),
\be\label{rho1m3}
|\rho_1'|  \leq  K \int_\R e^{-\ga |y|}z^2(t) + K\ve e^{-\ve\ga|\rho(t)|} \Big[  \int_\R e^{-\ga |y|}z^2(t) \Big]^{1/2} + K\ve^3,
\ee
and similarly for $|\tilde \rho_1'|$. Now we compare both identities, using (\ref{rho1m3}), to obtain,
\bea
  c_1' \Big[ \int_\R Q_c \partial_c \tilde u& &  - \int_\R  \Lambda Q_c z \Big]  - \tilde c_1' \Big[ \int_\R Q_{\tilde c} \partial_{\tilde c} \tilde v - \int_\R  \Lambda Q_{\tilde c} \tilde z\Big]   =  -   \int_\R Q_c \tilde S[\tilde u]  +   \int_\R Q_{\tilde c} \tilde S[\tilde v]  \label{C1a} \\
& &\qquad    - \rho_1' \Big[  \int_\R Q_c \partial_\rho \tilde u +\int_\R  Q_c' z \Big] + \tilde\rho_1' \Big[ \int_\R Q_{\tilde c} \partial_{\tilde \rho} \tilde v+ \int_\R  Q_{\tilde c}' \tilde z \Big] \label{C1} \\
&  &\qquad        + \ve f_1 \int_\R  \Lambda Q_c z   - \ve \tilde f_1 \int_\R  \Lambda Q_{\tilde c} \tilde z   \label{C2b}\\
 & &\qquad  -\ve^2 f_4 \int_\R  Q_c' z + \ve^2 \tilde f_4 \int_\R  Q_{\tilde c}' \tilde z + \ve^2 f_3\int_\R  \Lambda Q_c z   - \ve^2 \tilde f_3\int_\R  \Lambda Q_{\tilde c} \tilde z  \label{C3} \\
 & &\qquad   + \int_\R Q_c'  a(\ve x) [ 3\tilde u z^2 + z^3 ]  - \int_\R Q_{\tilde c}'  a(\ve x) [ 3\tilde v \tilde z^2 + \tilde z^3 ] \label{C5} \\
 & & \qquad + 3 \int_\R Q_c'  [a(\ve x)\tilde u^{2} - Q_c^{2}]z - 3 \int_\R Q_{\tilde c}'  [a(\ve x)\tilde v^{2} - Q_{\tilde c}^{2}]\tilde z. \label{C6}
\eea
We first deal with the right hand side of (\ref{C1a}). Since $Q_c$ and $Q_{\tilde c}$ are Schwartz functions, we have from (\ref{SIn3}),
$$
\abs{\int_\R Q_c \tilde S[\tilde u] - \int_\R Q_{\tilde c} \tilde S[\tilde v] }  \leq K \ve^3 e^{-\ve\ga|\rho(t)|} + K\ve^4.
$$
On the other hand,
\bee
\int_\R Q_c \partial_\rho \tilde u & = & \int_\R Q_c \partial_\rho \Big[ \eta_\ve(y) (\frac{Q_c(y)}{a^{1/2}(\ve \rho)} +\ve d(t) A_c(y) +\ve^2 B_c(t,y) ) \Big] \\
& =& - \int_\R  \eta_\ve Q_c \Big[ \ve a'(\ve \rho) \frac{Q_c(y)}{2a^{3/2}(\ve \rho)} + \frac{Q_c'(y)}{a^{1/2}(\ve \rho)}  +  \ve d(t) A_c'(y)  \Big]  + O(\ve^{2}) = O(\ve e^{-\ve\ga|\rho(t)|}) + O(\ve^2),
\eee
and similarly with the term $\int_\R Q_{\tilde c} \partial_{\tilde \rho} \tilde v$. Therefore, from (\ref{rho1m3}), we get
$$
|(\ref{C1})| \leq  K \ve^{1/2}  \int_\R e^{-\ga |y|}z^2(t) +K \ve^{1/2}  \int_\R e^{-\ga |\tilde y|}\tilde z^2(t) + K\ve^{5/2}.
$$
Now we deal with (\ref{C2b}). We have
$$
\abs{\ve f_1\int_\R \Lambda Q_c z -\ve \tilde f_1 \int_\R \Lambda Q_{\tilde c} \tilde z}  \leq K \ve e^{-\ga \ve|\rho(t)|} \Big[\int_\R z^2 e^{-\ga |y|} \Big]^{1/2}+ K \ve e^{-\ga \ve|\tilde \rho(t)|} \Big[\int_\R \tilde z^2 e^{-\ga |\tilde y|} \Big]^{1/2}. 
$$
The estimate of (\ref{C3}) is easy: from (\ref{INT41}) and (\ref{INT41v}), one has $|(\ref{C3})| \leq  K \ve^{5/2}$. Now we consider the terms in (\ref{C5}). We have
$$
 \abs{\int_\R Q_c'  a(\ve x) [ 3\tilde u z^2 + z^3 ]  - \int_\R Q_{\tilde c}'  a(\ve x) [ 3\tilde v \tilde z^2 + \tilde z^3 ] }   \leq K \Big[ \int_\R e^{-\ga|y|}z^2 + \int_\R e^{-\ga|\tilde y|}\tilde z^2\Big].
$$
Since $\|a(\ve x) \tilde u^2 -Q_c^2\|_{L^\infty(\R)} \leq K\ve e^{-\ga\ve |\rho(t)|} + K\ve^2,$ and similarly in the case of $\tilde v$, one has
$$
 |(\ref{C6})|  \leq K\ve e^{-\ga\ve |\rho(t)|}  \Big[ \int_\R e^{-\ga|y|} z^2\Big]^{1/2} +K \ve  e^{-\ga\ve |\rho(t)|}  \Big[ \int_\R e^{-\ga|\tilde y|} \tilde z^2\Big]^{1/2} + K\ve^{5/2}. 
$$
Finally, we deal with the left hand side of (\ref{C1a}):
$$
\int_\R Q_c \partial_c \tilde u -\int_\R \Lambda Q_c z = 2\theta \frac{c^{2\theta-1}(t)}{\tilde a(\ve\rho(t))} \int_\R Q^2 + O(\ve^{1/2})  \geq k_0 >0,\quad  (\theta = \frac 1{m-1}-\frac 14>0).
$$
Gathering the above estimates, we get finally (\ref{c1tc1}).
\end{proof}

\medskip



\begin{thebibliography}{10}

\bibitem{Benj} T.B. Benjamin, \emph{The stability of solitary waves}, Proc. Roy. Soc. London A \textbf{328}, (1972) 153 --183. 

\bibitem{BL} H. Berestycki and P.-L. Lions, \emph{Nonlinear scalar field equations. I. Existence of a ground state}, Arch. Rational Mech. Anal. \textbf{ 82}, (1983) 313--345.

\bibitem{BSS} J.L. Bona, P. Souganidis and W. Strauss,  \emph{Stability and instability of solitary waves of Korteweg-de Vries type}, Proc. Roy. Soc. London \textbf{411} (1987), 395--412.

\bibitem{SJ} S. I Dejak and B. L. G. Jonsson, \emph{Long-time dynamics of variable coefficient modified Korteweg-de Vries solitary waves}, J. Math. Phys. 47 (2006), no. 7, \textbf{072703}, 16 pp. 

\bibitem{DS} Dejak, S. I. and Sigal, I. M., \emph{Long-time dynamics of KdV solitary waves over a variable bottom}, Comm. Pure Appl. Math. \textbf{59} (2006), no. 6, 869--905.

\bibitem{GS} Gang, Z., and Sigal, I. M., \emph{Relaxation of solitons in nonlinear Schr\"odinger equations with potential}, Adv. Math. 216 (2007), no. 2, 443Ð490.
  
\bibitem{GW} Gang, Z., and Weinstein, M.I., \emph{Dynamics of Nonlinear Schr\"odinger / Gross-Pitaevskii Equations; Mass Transfer in Systems with Solitons and Degenerate Neutral Modes}, to appear in Anal. and PDE.

\bibitem{Gr1} R. Grimshaw, \emph{Slowly varying solitary waves. I. Korteweg - de Vries equation}, Proc. Roy. Soc. London Ser. A \textbf{368} (1979), no. 1734, 359--375.

\bibitem{GFJS} S. Gustafson, J. Fr\"ohlich, B. L. G. Jonsson, and  I. M. Sigal, \emph{Long time motion of NLS solitary waves in a confining potential}, Ann. Henri Poincar\'e \textbf{7} (2006), no. 4, 621--660.

\bibitem{FGJS} S. Gustafson, J. Fr\"ohlich, B. L. G. Jonsson, and I. M. Sigal, \emph{Solitary wave dynamics in an external potential}, Comm. Math. Phys. \textbf{250} (2004), 613--642. 

\bibitem{H} J. Holmer, \emph{Dynamics of KdV solitons in the presence of a slowly varying potential}, to appear in IMRN.

\bibitem{HZ} J. Holmer, and M. Zworski, \emph{Soliton interaction with slowly varying potentials}, Int. Math. Res. Not., (2008) \textbf{2008}, art. ID rnn026, 36 pp.

\bibitem{HMZ0}  J. Holmer, J. Marzuola, and M. Zworski, \emph{Soliton Splitting by External Delta Potentials}, J. Nonlinear Sci. \textbf{17} no. 4 (2007), 349--367.

\bibitem{HMZ} J.  Holmer, J. Marzuola and M. Zworski, \emph{Fast soliton scattering by delta impurities}, Comm. Math. Phys., \textbf{274}, no.1 (2007) 187--216.

\bibitem{KM1} V.I. Karpman and E.M.  Maslov, \emph{Perturbation theory for solitons},
Soviet Phys. JETP \textbf{46} (1977), no. 2, 537--559.; translated from Z. Eksper. Teoret. Fiz. \textbf{73} (1977), no. 2, 281--29.

\bibitem{KN1} D. J . Kaup, and A. C. Newell, \emph{Solitons as particles, oscillators, and in slowly changing media: a singular perturbation theory}, Proc. Roy. Soc. London Ser. A \textbf{361} (1978), 413--446.

\bibitem{KK} K. Ko and H. H. Kuehl, \emph{Korteweg-de Vries soliton in a slowly varying medium}, Phys. Rev. Lett. \textbf{40} (1978), no. 4, 233--236.

\bibitem{KPV} C.E. Kenig, G. Ponce and L. Vega, \emph{Well-posedness and scattering results for the generalized Korteweg--de Vries equation via the contraction principle}, Comm. Pure Appl. Math. \textbf{46}, (1993) 527--620. 

\bibitem{Lo} P. Lochak, \emph{On the adiabatic stability of solitons and the matching of conservation laws}, J. Math. Phys. \textbf{25} (1984), no. 8, 2472--2476.

\bibitem{Martel} Y. Martel, \emph{Asymptotic $N$--soliton--like solutions of the subcritical and critical generalized Korteweg--de Vries equations}, Amer. J. Math.  \textbf{127} (2005), 1103--1140.

\bibitem{MMnon} Y. Martel and F. Merle, \emph{Asymptotic stability of solitons of the subcritical gKdV equations revisited}, Nonlinearity \textbf{18} (2005) 55--80.

\bibitem{MMcol1} Y. Martel and F. Merle, \emph{Description of two soliton collision for the quartic gKdV equations}, to appear in Annals of Math.

\bibitem{MMcol2} Y. Martel and F. Merle, \emph{Stability of two soliton collision for nonintegrable gKdV equations}, Comm. Math. Phys. \textbf{286} (2009), 39--79.

\bibitem{MMfin2} Y. Martel and  F. Merle,  \emph{Inelastic interaction of nearly equal solitons for the quartic gKdV equation}, to appear in Inventiones Mathematicae. 

\bibitem{MMT} Y. Martel, F. Merle and T. P. Tsai, \emph{Stability and asymptotic stability in the energy pace of the sum of $N$ solitons for subcritical gKdV equations}, Comm. Math. Phys. \textbf{231} (2002) 347--373.

\bibitem{Mu1} C. Mu\~noz, \emph{On the soliton dynamics under slowly varying medium for generalized NLS equations}, to appear in Math. Annalen (arXiv:1002.1295). 

\bibitem{Mu2} C. Mu\~noz, \emph{On the soliton dynamics under slowly varying medium for generalized KdV equations}, to appear in Anal. \& PDE (arXiv:0912.4725).

\bibitem{Mu3} C. Mu\~noz, \emph{On the soliton dynamics under slowly varying medium for generalized KdV equations: refraction vs. reflection}, preprint arXiv:1009.4905.

\bibitem{Mu5} C. Mu\~noz, \emph{Dynamics of soliton-like solutions for slowly varying, generalized KdV equations}, Oberwolfach report 2010.

\bibitem{New} A. Newell, \emph{Solitons in Mathematics and Physics}, CBMS-NSF Regional Conference Series in Applied Mathematics, 48. Society for Industrial and Applied Mathematics (SIAM), Philadelphia, PA, 1985. 

\bibitem{PW} R. L. Pego, and M.I. Weinstein, \emph{Asymptotic stability of solitary waves}, Comm. Math. Phys. \textbf{164} (1994), no. 2, 305--349. 

\bibitem{P} G. Perelman, \emph{A remark on soliton-potential interactions for nonlinear Schr\"odinger equations},  Math. Res. Lett. \textbf{16} 3 (2009), pp. 477--486.


\bibitem{We1} M.I. Weinstein, \emph{Modulational stability of ground states of nonlinear Schr\"odinger equations}, SIAM J. Math. Anal. \textbf{16} (1985), no. 3, 472Ð491.


\end{thebibliography}
\end{document}